\documentclass[10pt]{amsart}

\usepackage[a4paper]{geometry}
\usepackage[T1]{fontenc}
\usepackage[utf8]{inputenc}
\pdfoutput=1

\usepackage{graphicx}
\usepackage{amsmath,amsthm, amsfonts, amssymb}

\newtheorem{theorem}{Theorem}
\newtheorem*{theorem*}{Theorem}
\newtheorem{lemma}{Lemma}[section]
\newtheorem{proposition}[lemma]{Proposition}

\theoremstyle{definition}
\newtheorem{definition}[lemma]{Definition}
\theoremstyle{remark}
\newtheorem{remark}{Remark}[section]

\newcommand{\naturales}{\mathbb{N}}
\newcommand{\real} {\mathbb{R}}
\newcommand{\enteros} {\mathbb{Z}}

\newcommand{\R}{{\bf\sf R}}

\newcommand{\diam}{\mbox{diam}}
\newcommand{\Vol}{\mbox{Vol}\, }
\newcommand{\dH}{\dim_{\mathcal H}\,}
\newcommand{\dd}{\dim_{\mathcal D}\,}
\newcommand{\imply}{\Rightarrow}

\title{New bounds for the Hausdorff dimension \\ of a dynamically defined Cantor set}

\author{Fernando J. S\'anchez-Salas}
\address{Departamento de Matem\'aticas, Facultad Experimental de Ciencias, Universidad del Zulia, Avenida Universidad, Edificio Grano de Oro, Maracaibo, Venezuela}

\email{fjss@fec.luz.edu.ve}

\date{March 14, 2023}

\begin{document}
\begin{abstract}
 In this paper we use the additive thermodynamic formalism to obtain bounds of the Hausdorff and box-counting dimension of certain non conformal hyperbolic Cantor sets defined by piecewise smooth expanding maps on a $d$-dimensional smooth manifold $M$.
\end{abstract}

\maketitle

\section{Introduction}  

The calculation of dimensional characteristics as the Hausdorff and box-counting dimensions is a key topic of fractal geometry of sets and measures. Theses quantities are indicators of the size and complexity of fractal sets. See \cite{falconer} and \cite{mattila}. In this paper we study some aspects of the fractal geometry of non conformal hyperbolic Cantor sets left invariant by piecewise $C^r$, $r > 1$ differentiable expanding maps on a compact Riemannian manifold with dimension $d > 2$. Concretely we use the additive thermodynamic forma\-lism to obtain lower and upper bounds for the Hausdorff and box-counting dimension of dynamically defined Cantor sets. 
We refer to \cite{barreira.gelfert.2011} and \cite{pesin.review.2010} for an overview of problems and results on the subject of dimension theory of dynamical systems. The problem to find lower bounds for the Hausdorff dimension of invariant sets of certain dynamical systems it is interesting in its own and had been discussed previously for 
attractors of dissipative reaction-diffusion equations in \cite{efendiev.miranville.2003}, Rauzy gaskets \cite{gutierrez.2020}, Bernoulli convolutions \cite{kleptsyn.2021}, attractors of certain dynamical systems \cite{li.1995}, geometric Lorenz attractor in \cite{clizana.2008} and \cite{mora.2012} and $n$-dimensional self-affine sets \cite{paulsen.1995}. We will deal with dynamically defined Cantor sets defined as maximal $f$-invariant sets of a piecewise expanding maps in higher dimensional manifold. This are the simplest models of uniformly hyperbolic repellers. See section \ref{section.setting} below. Our approach will rely upon the volumetric method used in \cite{vanderlei.viana} to give upper bounds for the Hausdorff dimension of certain nonuniformly expanding repeller. See section \ref{section.methods}. Our new bounds corrects an error in our lower bound in Proposition 5.1 \cite{sanchez.salas.2002} pointed out in \cite{cao.wang.zhao.2023}. This proposition were used to prove that a Sinai-Bowen-Ruelle measure can be approximated by sequences of horseshoes with arbitrarily large unstable dimension, a result previously proved for surface diffeomorphisms in \cite{mendoza.1988}. The present correction solves an essential error in the proof of Proposition 5.1, however Main Theorem in \cite{sanchez.salas.2002} may require further modifications which may appear in a manuscript in preparation.

Besides this correction our present result points to further applications we have in mind. Namely, in
the work of Buzzi, Crovisier and Sarig \cite{buzzi.crovisier.sarig} on finiteness of equilibrium states for surface diffeomorphisms, dimension theory of dynamical systems is used to show that, if topological entropy is sufficiently large then suitable horseshoes associated with hyperbolic measures converging to an equilibrium state are sufficiently "thick" in such way that they are homoclinically related, that is, there exists a transversal intersection between their stable and unstable manifolds. In this argument it is crucial to relate the topological entropy to the thickness of the approximating horseshoes, relying upon the dimension theory of surface diffeomorphisms. Then we rise the following question: \textit{there exists, in the higher dimensional setting, another dynamical indicator, involving the topological entropy or other dynamically defined quantities, such that if this quantity is sufficiently large then the approximating horseshoes used on \cite{buzzi.crovisier.sarig} are thick enough and then homoclinically related, as in the two dimensional setting?} Our present results suggest that this question can be answered in the affirmative. Concretely, volumetric methods can be used to construct a suitable dynamical indicator, involving topological entropy, Lyapunov exponents, average rates of volume and length expansions along invariant manifolds which controls how large are the stable and unstable Hausdorff dimensions of non conformal uniformly hyperbolic horseshoes, pointing to a finiteness result of equilibrium states in the higher dimensional setting, using Buzzi, Crovisier and Sarig arguments. This is part of a work in progress which will appear elsewhere.

\section{The setting}\label{section.setting}
\subsection{Cantor sets as limits of regular constructions}\label{subsection.regular.constructions}
We are interested in the box-counting and Hausdorff dimensions of the limit set of certain geometrical constructions. Let us describe our basic model. Let $N \subset M$ be a compact domain with non empty interior. What we have in mind is basically a closed $d$-cube $[0,1]^d$ embedded in $M$ with piecewise smooth border $\partial{N}$, where $d = \dim(M)$. We also need finitely many one-to-one maps $g_i :  N \to N$, $i = 1, \cdots , s$ satisfying the following conditions:
\begin{itemize}
 \item the maps $g_i$, $i = 1, \cdots s$ are $C^r$ ($r > 1$) \textbf{contractive}: 
 $$
 \|Dg\| := \max_{1 \leq i \leq s}\sup_{x \in U}\|Dg_i(x)\| < 1;
 $$
 
\item \textbf{open condition}: there exists an open neighborhood $U \supset N$ such that $U_i = g_i(U) \subset int(N)$ are pairwise disjoint: $U_i \cap U_j = \emptyset$ for $i \neq j$; 

\item \textbf{border condition}: we will suppose in addition that 
$\partial{N_i} \cap \partial{N} = \emptyset$.
\end{itemize}

It is not usual to ask for a border condition as above, however we need it in some parts of our arguments. It is not too restrictive as long as we can approximate a generic geometric construction by a configuration satisfying this property. See below.

By a well known theorem due to Hutchinson, there exists a limit set
$$
\Lambda = \bigcap_{n = 1}^{+\infty}\Lambda_n \quad\text{where}\quad \bigcup_{(i_1, \dots , i_n) \in s^n}g_{i_1} \circ \cdots \circ g_{i_n}(N).
$$
The set $\Lambda_n$ is the $n$-th stage of the construction and $\Lambda$ the attractor set of the IFS ${\bf g} = \{g_i\}$ acting, as a set map, upon the (complete) metric space of compact subsets of $N$ with the Hausdorff distance, where ${\bf g}(A) = \bigcup_ig_i(A)$ is the set map and ${\bf g} \circ {\bf g}(A) = \bigcup_{i,j}g_i \circ g_j(A)$, the composition law. The collection $\wp_n  = \{g_{i_1} \circ \cdots \circ g_{i_n}(N): (i_1, \dots , i_n) \in s^n\}$ is also known as a \textit{Markov partition} of the $n$-stage of the construction, which will be called \textit{atoms of generation $n$}: $\Lambda_n = \bigcup_{P \in\wp_n}P$. The family of sets of first generation is the \textit{template} of the construction. Topologically, the attractor of a geometric construction $\Lambda$ satisfying the open and border condition is a Cantor set, that is, a compact, totally disconnected, perfect set. In particular $\Lambda$ has zero topological dimension. Furthermore, our method requires that there exists a well defined bounded relation between diameter and volume inside the set, in the following sense.

\begin{definition}
We say that a subset $X \subset M$ is \textit{volume reducible} if there exist some $k < d$ such that $X$ can be covered with \textit{at most countably many smooth subsets $B_i$ of positive $k$-volume}, where 
$\Vol_k(B_i) \asymp radius(B)^{k}$  uniformly. Here $d$ is the dimension of $M$, the ambient manifold. By smooth we mean that $B_i \subset M$ is an smoothly embedded $k$-disc. A subset is \textit{volume irreducible} if it is not volume reducible.
\end{definition}

A cartesian product $X \times Y$ of Cantor sets $X,Y \subset \real$ is irreducible since it can not be covered by countably many intervals. Similarly so a Cantor set $\Lambda \subset \real^d$ whose projections onto the every $k$ dimensional coordinate plane is a Cantor set. Moreover, if some of these projections have positive Lebesgue measure, the set is still volume irreducible, whenever there exits at least one Cantorian projection. A sufficient condition for a limit set $\Lambda$ to be volume reducible is that there exists an embedded smooth submanifold $Z \subset M$ of codimension $\geq 1$ such that 
$\Lambda \subset Z$. This situation appears in some examples of self-affine subsets, depending of the initial configuration of the template. Volume reducible configurations are exceptional in that this condition can be removed by an small perturbation of the initial template of the IFS. 

If the construction converges to a non rectifiable connected set, as in the case of Weierstrass-like graphs, the limit set it is not volume reducible in the sense of our definition, since it can not be covered by countably many \textit{smooth} $k$-discs, with $k < d$. However, falls outside of reach for our methods there is not a well defined notion of smooth $k$-dimensional volume for these sets. In some cases, we can approximate these graphs-like curves or connected sets by irreducible Cantor sets, as we will see below, when discussing Bedford-McMullen carpets example. Then, we will suppose, for the sake of simplicity in our arguments, that the limit of our constructions are volume irreducible Cantor sets in the $d$ dimensional ambient 
manifold $M$. 

\begin{definition}\label{definition.regular.constructions}
We say that a contractive Iterated Function System (IFS) $\{g_i\}$ satisfying the open and border condition, whose limit set is volume irreducible, is a \textit{regular geometric construction}.
\end{definition}

\subsection{The dimension theory of conformal sets}\label{subsection.conformal.sets}
Let
\begin{equation}\label{piecewise.map.definition}
f : \bigcup_iN_i \to N \subset M
\end{equation}
be a piecewise expanding $C^r$ ($r > 1$) map on a $d$-dimensional compact Rie\-mannian manifold with dimension $d \geq 1$, where $N \subset M$ is a compact $d$-cube and $N_i \subset$ are finitely many non overlapping $d$-cubes placed inside $N$ and $f|N_i : N_i \to N$ is an expanding diffeomorphism for each $i$ and
\begin{equation}\label{maximal.invariant.set}
\Lambda = \bigcap_{n=0}^{+\infty}f^{-n}\left(\bigcup_iN_i\right)
\end{equation}  
be its maximal invariant subset. If $d = 1$ then $f$ is simply a piecewise expanding interval map.  If $d > 1$ we say that $f$ is \textit{conformal} if there exists a non zero function $a = a(x)$ such that 
$Df(x) = a(x)id_{T_xM}$. If $f$ is an interval or conformal piecewise expanding map, the following results are well known:
\begin{enumerate}
 \item $\dH(\Lambda) = \dim_B(\Lambda)$;
 \item $\dH(\Lambda) = d\alpha$, where $0 < \alpha \leq 1$ is the unique solution to the Bowen equation 
 $P(f|\Lambda,-\alpha\log{Jf}) = 0$, where $Jf = |a(x)|^d$ is the Jacobian wrt the Riemannian volume;
 \item the Hausdorff measure is finite and positive $0 < \mathcal{H}_{d\alpha}(\Lambda) < +\infty$;
 \item there exists a unique Borel probability of maximal dimension $\mu_{\Lambda}$, such that:
 $$
 \dH(\Lambda) = \dfrac{h(\mu_{\Lambda})}{\chi^+(\mu_{\Lambda})} = \sup_{\mu \in \mathcal{M}_f(\Lambda)}\dfrac{h(\mu)}{\chi^+(\mu)},
 $$
 where 
 $$
 \chi^+(\mu) = \int_{\Lambda}\sum\limits_{\chi_i^+(x) > 0}\chi_i^+(x)d\mu(x) 
 $$
 is the sum of the positive Lyapunov exponents of the $f$-invariant probability $\mu$;
 \item the Hausdorff dimension $\dH(\Lambda_f)$ varies smoothly with $f$. See Appendix A and references therein.
\end{enumerate}

Conformal hyperbolic repellers as the ones described above are the simplest models where (additive) thermodynamic formalism can be used sucessfully to calculate geometrical dimensions in terms of the dynamical indicators. In our case these repellers are topologically (semi)conjugated to a full shift with finitely many symbols. More generally, hyperbolic repellers are (semi)conjugated to subshifts of finite type. And there are also important nonuniformly expanding examples, as the one studied in \cite{vanderlei.viana}.

\subsection{Statement of the problem}
The solution of the Bowen equation is the \textit{dynamical dimension}, that is, a dimensional characteristic of the dynamics. We are interested in to understand its connection with geometrical dimensions as the box-counting dimension and the Hausdorff dimension. We refer to subsection \ref{subsection.dimension.theory} in Appendix A to recall the definitions of the main fractal dimensions and subsection \ref{subsection.thermodynamics.dimension} to recall the connection between the thermodynamic formalism and dimension theory of hyperbolic repellers.

When we are dealing with interval maps the connection between geometrical and dynamical dimensions comes from the fact that volume and diameter of an interval are the same. A natural question arises as to \textit{why, in the higher dimensional setting, the dynamical and geometrical dimensions of conformals sets coincide}? The simplest explanation is that the atoms of the construction $P \in \wp_n$ at every stage of the construction have bounded \textit{geometrical distortion}. Indeed, let
$$
r_{inner}(P) = \sup\{r > 0 : \exists \ x \in P, \ B(x,r) \subset P\}
$$
$$
r_{outer}(P) = \inf\{r > 0 : \exists \ x \in P, \ P \subset B(x,r)\}
$$
be the inner and outer diameters of the set $P$, respectively. It can be proved that, if $f$ is conformal, then there exists a suitable constant $C > 1$ such that 
$$
C^{-1} \leq \dfrac{r_{inner}(P)}{r_{outer}(P)} \leq C, \quad \forall \ P \in \wp_n, \ n \geq 0.
$$ 
In other words, the geometrical distortion of the generating sets is uniformly bounded. Therefore, in these cases the dynamical measure $\mathcal{D}_{a}$, which is defined in terms of coverings by atoms $P \in \wp_n$ of $n$-th generatins, is equivalent to the Hausdorff measure $\mathcal{H}_{da}$, for every $a > 0$, defined by coverings by metric balls with (possibly different) radius $< \zeta$, also called $\zeta$-coverings (see Appendix A). Consequently, the dynamical and the Hausdorff dimension are equal, up the multiplicative constant $d$ having all the nice properties stated at the end of subsection \ref{subsection.conformal.sets}.

The geometrical distortion of the atoms of a non conformal construction is unbounded. That makes additive thermodynamics a less useful tool to calculate the fractal dimensions. Indeed, coverings by atoms $O \in \wp_n$ of $n$-th generation of the construction are not efficient to estimate the Hausdorff measure. This sets a number of interesting and difficult problems which had grow into new methods and results. The most used approach is to generalize the thermodynamic formalism for subadditive or almost additive families of potentials, generating a new class of pressure-like dynamical indicators. We refer to the book of Barreira \cite{barreira.dimension.2011} for a comprehensive exposition of these methods, applied to the dimension theory of dynamical systems and \cite{barreira.gelfert.2011} and \cite{pesin.review.2010} for a review of results and open problems. We also need to mention the study non conformal self-affine sets, the simplest geometric models of a non conformal limit set. This has arised a number of formulas and bounds for the Hausdorff dimension for several models as Bedford-McMullen and Gatzouras-Lalley carpets, among others. We refer to the survey of Falconer \cite{falconer.survey.2013} for results, examples and references on the subject. 

The aim of this paper is to show that the standard additive thermodynamic formalism and the Bowen equation can still be used to get useful upper and lower bounds for these dimensional indicators in the non conformal setting.

\section{Statement of results}
Let $\mathcal{G} = \{g_i\}$ be an IFS defining a regular construction as the one defined in (\ref{definition.regular.constructions}) and $f: \bigcup_iU_i \to U$ be the piecewise expanding map defined by the inverses $f|U_i = g_i^{-1}$ of the contracting branches. Let $N \subset U$ be a $d$-dimensional compact cube and
$f : \bigcup_iN_i \to N \subset M$ the corresponding piecewise expanding map and $\Lambda = \bigcap_{n=0}^{+\infty}f^{-n}\left(\bigcup_iN_i\right)$ be the maximal $f$-invariant Cantor set, that is, $\Lambda$ is the attractor of the regular construction $\mathcal{G}$. Before to state our main result we need to introduce the following quantities:
\begin{itemize}
 \item we define
 \begin{equation}\label{defining.|Jg|}
\|\underline{Jg}\|   =  \min_{i = 1, \cdots , s}\inf_{x \in N_i}Jg_i(x)
 \end{equation}
and
\begin{equation}\label{defining.|Dg|}
 \|Dg\|   = \max_{i = 1, \cdots , s}\sup_{x \in N_i}\|Dg_i(x)\|
\end{equation}
 where $Jg_i(x) = |\det(Dg_i(x))|$ is the Jacobian with respect to the Riemannian volume and $Dg_i(x)$ the derivative at $x$;
 \item define $0 < \epsilon < 1$ as:
 \begin{equation}\label{defining.epsilon}
  \epsilon = \dfrac{\|\underline{Jg}\|}{\quad\quad\|Dg\|^{d - 1}} \leq \|Dg\| < 1.
 \end{equation}
The number $\epsilon > 0$ is a characteristic scale at which every $n$-atom $P \in \wp_n$ can be uniformly covered by open $2\epsilon^n$-balls and their volume are uniformly comparable, independently of $n > 0$. These special coverings play the role of decompositions of every $P \in \wp$ into pieces with uniformly bounded geometrical distortion. 
\item We also denote
\begin{equation}\label{defining.lambda0}
\lambda_0 = \log\left(\inf_{x \in \Lambda}Jf(x)\right) 
\end{equation}
and
\begin{equation}\label{defining.lambda1}
\lambda_1 = \log\left(\sup_{x \in \Lambda}Jf(x)\right), 
\end{equation}
where $Jf(x) = |\det(Df(x)|$ is the Jacobian and
$$
\Lambda = \bigcap_{n=0}^{+\infty}f^{-n}(N)
$$
is the maximal $f$-invariat subset in $N$;
\item finally, we let
\begin{equation}\label{defining.|Df|}
\|Df\| = \max\limits_{i}\sup_{x \in U_i}\|Df(x)\|
\end{equation}
the largest rate of expansion of the length under $f$. 
\end{itemize}

We will prove the following 

\begin{theorem}\label{main.thm.1}
With the above assumptions and definitions of constants we have that:
 \begin{equation}\label{main.inequality}
  L \leq \dH(\Lambda) \leq \underline{\dim_B}(\Lambda) \leq \overline{\dim_B}(\Lambda) \leq U,
 \end{equation}
 where $\dH(\Lambda)$ is the Hausdorff dimension,  $\underline{\dim_B}(\Lambda)$ (resp. $\overline{\dim_B}(\Lambda)$) the lower (resp. upper) box-counting dimension and 
 $$
 U = d + \dfrac{\lambda_0(1-\alpha)}{\log\epsilon}
 \quad\text{and}\quad
 L = d + \dfrac{\overline{\Sigma} + \lambda_1(1-\alpha)}{\log\epsilon}.
 $$
 $\alpha$ is the solution to the Bowen equation $P(f|\Lambda,-\alpha\log{Jf}) = 0$ and 
 $$
 \overline{\Sigma} = \log\left((\|Df\|\epsilon)^{-\overline{d_B(\Lambda)} + d}\|Df\|^{-d}e^{\lambda_0}\right)^{-1}.
 $$ 
\end{theorem}

The constant $\overline\Sigma$ is an upper bound for a function $\Sigma(x)$, which measures the rate of exponential decay as $n \to +\infty$, of the volume fraction of an atom $\wp_n(x)$ containing $\Lambda$, covered by a suitable efficient covering of $\Lambda \cap \wp_n(x)$ by $2\epsilon^n$-balls.

As a consequence of the thermodynamics, the bounds in Theorem \ref{main.thm.1} are robust and vary smoothly under small $C^1$ perturbations of the construction. This is a plain consequence of the smoothness of the additive pressure $P(\phi)$ with respect to the potential $\phi$ and the robustness and structural stability of hyperbolic repellers.

Theorem \ref{main.thm.1} answer partially a question possed to me by Marcelo Viana during the preparation of my doctoral dissertation at IMPA. He asked me what kind of relationship exists between the dynamical and Hausdorff dimension. The quantities $L$ and $U$  are \textit{are dynamical indicators} and they point to a precise functional relationship between these dimensional indicators and the box-counting and the Hausdorff dimension. In fact, the functions
$$
F(t,s) = d + \dfrac{t + \lambda_1(1-s)}{\log\epsilon} \quad\text{and}\quad G(t, \alpha) = d + \dfrac{t + \lambda_0(1-s)}{\log\epsilon} \quad t \in \real, \ 0 < s \leq 1,
$$
are affine stricly decreasing functions of $t$. Then, one can conjecture that there exists a unique $\Sigma_H = \Sigma(\dH(\Lambda)) \leq \overline{\Sigma}$ and $\Sigma_b = \Sigma(\overline{\dim}_B(\Lambda)) \geq 0$ such that
$$
\dH(\Lambda) = d + \dfrac{\Sigma_H + \lambda_1(1-\alpha)}{\log\epsilon} = F(\Sigma_H, \alpha)
$$
and
$$
\overline{\dim}_B(\Lambda) = d + \dfrac{\Sigma_b + \lambda_0(1-\alpha)}{\log\epsilon} = G(\Sigma_b, \alpha),
$$
where $\alpha$ is the solution to Bowen's equation. In other words both, there exists an explicit functional relation between the Hausdorff and box-counting dimension and the dynamical dimension. Unfortunately, at the present we don't know how to compute the constants $\Sigma_H$ and $\Sigma_b$ in contrast with our upper and lower bounds, which can be explicitly calculated from the data of the construction, as we will see below.

\section{Two non conformal examples}\label{section.discussion}

We will start our study with the simplest non conformal models as a test ground for our methods, letting the study of more general cases for future indagations. For this we applicate our results to the non conformal plane self-affine model discussed by Pollicot and Weiss in \cite{pollicot.weiss.1994} and Bedford-McMullen carpets. See \cite{falconer.survey.2013} for a survey of results of dimension theory of plane self affine sets. 

\subsection{Pollicot-Weiss model:} let $0 < \tau < \beta$ with $\tau, \beta \neq 1/2$, be real numbers and consider the couple of affine contractions:
$$
g_0(x,y) = \left[ 
           \begin{array}{cc}
            \beta  & 0 \\
             0     & \tau
           \end{array}
           \right]\left[\begin{array}{c}
                         x \\
                         y
                        \end{array}
                  \right]                  
                  \quad\text{and}\quad 
                  g_1(x,y) = \left[ 
           \begin{array}{cc}
            \beta  & 0 \\
             0     & \tau
           \end{array}
           \right]\left[\begin{array}{c}
                         x \\
                         y
                        \end{array}
                  \right] + \left[\begin{array}{c}
                         1 - \beta\\
                         1 - \tau
                        \end{array}
                  \right]
$$
This generates a limit set $\Lambda$ starting with two rectangles $\R_0$ and $\R_1$ of width $\beta$ and height $\tau$ attached to the down-left and upper-right corners of unit rectangle, respectively. This is the simplest non conformal self affine limit set. If $\tau < 1/2$ the rectangles are disjoint and the limit is a plane non conformal Cantor set. These Cantor sets are volume-irreducible. Degenerate configurations in which rectangles $\R_i$ are situated above a single interval in the $OX$ axis, are excluded; in this case $\Lambda$ is contained in a single vertical line and it is volume reducible, as it is easy to see. If $\tau = 1/2$ the rectangles intersect at the border and the construction converges to the graph of a continuous non differentiable function, unless $\beta = 1/2$, in which case the set is conformal and the limit is the diagonal $x = y$, at least with this choice of translations. A remarkable feature of this example is that, for almost every $\beta$ such that 
$0 < \tau < 1/2 < \beta$, the projection of $\Lambda$ onto the $OX$ axis is an absolutely continuous Bernoulli convolution. See \cite{falconer.survey.2013} and \cite{pollicot.weiss.1994}. For this particular choice of parameters, Hausdorff and box-counting dimension counting dimension are equal.

Notice also that the construction does not satisfy our border condition. However, we can separate the rectangles from the border with the same arguments used previously for Bedford-McMullen carpets. Certainly these perturbations modify the number-theoretical properties which are responsible for the remarkable geometry of these examples. However, as we are arguing, our goal is to give bounds for the Hausdorff dimension. We are going to use these exact formulas as a test case for our inequalities. Let us begin computing the constants involved in our formulas. 
$$
Dg_i = \begin{bmatrix}
        \beta & 0 \\
        0 & \tau
       \end{bmatrix},
\quad
f|\R_0 = 
 \begin{bmatrix}
  \beta^{-1} & 0\\
  0 & \tau^{-1}
 \end{bmatrix}
 \begin{bmatrix}
  x \\
  y
 \end{bmatrix}
\quad\text{and}\quad
f|\R_1 = 
 \begin{bmatrix}
  \beta^{-1} & 0\\
  0 & \tau^{-1}
 \end{bmatrix}
 \begin{bmatrix}
  x+\beta-1 \\
  y+\tau-1
 \end{bmatrix}.
$$
Then
$$
Df = \begin{bmatrix}
        \beta^{-1} & 0 \\
        0 & \tau^{-1}
       \end{bmatrix}
\quad\text{and}\quad
\|Dg\| = \beta, \quad \|\underline{Jg}\| = \tau\beta, \quad \|Df\| = \tau^{-1} \quad\text{and}\quad Jf = (\beta\tau)^{-1}
$$
Thus,
$$
\lambda_0 = \lambda_1 = \log(Jf) = \log((\beta\tau)^{-1}) = -\log(\beta\tau).
$$ 
On the other side,
$$
\epsilon = (1-\delta)\left(\dfrac{\|\underline{Jg}\|}{\|Dg\|^{d - 1 + \delta}}\right)^{\frac{1}{1-\delta}} = (1-\delta)\left(\dfrac{\beta\tau}{\beta^{1 + \delta}}\right)^{\frac{1}{1-\delta}} = (1-\delta)(\beta^{-\delta}\tau)^{\frac{1}{1-\delta}}.
$$
In particular,
$$
\delta = 0 \quad\imply\quad \epsilon = \tau.
$$
To calculate the solution to the Bowen equation we argue as follows. As the Jacobian $Jf$ is constant we have
$$
-s\sum_{k=0}^{n-1}\log(Jf) = \log[Jf]^{-sn}
$$
and thus
$$
\sum_{(i_1, \cdots , i_n) \in 2^n}\exp\left(-s\sum_{k=0}^{n-1}\log(Jf)\right) = \sum_{(i_1, \cdots , i_n) \in 2^n}[Jf]^{-sn} = 2^n[Jf]^{-sn} = (2Jf^{-s})^n. 
$$
Therefore,
$$
\limsup_{n \to +\infty}\dfrac{1}{n}\log\left(\sum_{(i_1, \cdots , i_n) \in 2^n}\exp\left(-s\sum_{k=0}^{n-1}\log(Jf)\right)\right) = \log(2Jf^{-s}) = 0
$$
if and only if $2Jf^{-s} = 1$ so giving: 
$$
\alpha = \dfrac{\log{2}}{\log{Jf}} = -\dfrac{\log{2}}{\log(\beta\tau)}
$$
for the dynamical dimension and
$$
1 - \alpha = 1 + \dfrac{\log{2}}{\log(\beta\tau)} = \dfrac{\log(2\beta\tau)}{\log(\beta\tau)}.
$$
Then, our upper bound for the box-counting dimension for $\delta = 0$ is
\begin{eqnarray*}
U  & = & 2 + \dfrac{\lambda_0(1-\alpha)}{\log(\epsilon)}
             =  
2 + \dfrac{\log((\beta\tau)^{-1})}{\log(\tau)}
\dfrac{\log(2\beta\tau)}{\log(\beta\tau)}\\
& = & 2-\dfrac{\log(2\beta\tau)}{\log(\tau)} = 1 - \dfrac{\log(2\beta)}{\log\tau} \\
& = & -\dfrac{\log(2\beta/\tau)}{\log\tau} = \dfrac{\log(2\beta/\tau)}{\log(1/\tau)}.
\end{eqnarray*}
Concluding that \textit{for every $0 < \tau < \beta < 1$ and $\tau < 1/2$}
$$
\dH(\Lambda) \leq \underline{\dim}_B(\Lambda) \leq \overline{\dim}_B(\Lambda) \leq \dfrac{\log(2\beta/\tau)}{\log(1/\tau)},
$$
excluding degenerate configurations leading to volume reducible sets.

This bound can be also obtained arguing directly. In the $n$-th stage of the construction we have $2^n$ rectangles,
$$
R(i_0, \cdots , i_{n-1}) = g_{i_{n-1}} \circ \cdots g_{i_0}(R).
$$
Each one of them can be decomposed into $(\beta/\tau)^m$ squares of side $\tau^m$, $m \geq n$. This forms a collection $\mathcal{B}_m = \{R(\tau^m)\}$. The number of these rectanbles are
$$
\#\mathcal{B}_m = 2^n(\beta/\tau)^m.
$$
If $m = n$ this forms a minimal covering set of the limit set $\Lambda$ by $\tau^n$-squares. Then, 
$\mathcal{N}(\Lambda, \tau^n) \leq 2^n(\beta/\tau)^n$ and
$$
\limsup_{n \to +\infty}\dfrac{\log(\mathcal{N}(\Lambda, \tau^n))}{\log(1/\tau^n)} \leq \dfrac{\log(2\beta\tau)}{\log(1/\tau)},
$$
so proving that, for every $0 < \tau < \beta < 1$, 
$$
\dH(\Lambda) \leq \underline{\dim}_B(\Lambda) \leq \overline{\dim}_B(\Lambda) \leq \dfrac{\log(2\beta\tau)}{\log(1/\tau)}
$$
Then our formula gives good upper bounds for this collection of examples. Now we turn to the estimation of lower bound. For this we need to calculate
$$
\overline\Sigma = \log((\|Df\|\epsilon)^{-\overline{d_B(\Lambda)} + d}\|Df\|^{-d}e^{\lambda_0})^{-1}.
$$
As $\|Df\| = \tau^{-1}$ and $\epsilon = \tau$ we get
$$
\overline\Sigma = -\log(\tau^2(\beta\tau)^{-1}) = -\log(\tau/\beta) = \log(\beta/\tau).
$$
Therefore,
\begin{eqnarray*}
L = 2 + \dfrac{\overline\Sigma + \lambda_1(1-\alpha)}{\log\epsilon} & = & 2 + 
\dfrac{\log\left(\frac{\beta}{\tau}\right) - \log(\beta\tau)\left(1 + \frac{\log(2)}{\log(beta\tau}\right)}{\log(\tau)}\\
& = & 2 + \dfrac{\log\left(\frac{\beta}{\tau}\right) - \log(2\beta\tau)}{\log(\tau)}\\
& = & 2 -\dfrac{2\log(\tau) + \log(2)}{\log(\tau)} = 2 + \dfrac{\log(2\tau^2)}{\log\left(\frac{1}{\tau}\right)}\\
& = & \dfrac{\log(2\tau^2) - 2\log(\tau)}{\log\left(\frac{1}{\tau}\right)} = \dfrac{\log(2)}{\log\left(\frac{1}{\tau}\right)}.
\end{eqnarray*}
Hence,
$$
L = \dfrac{\log(2)}{\log\left(\frac{1}{\tau}\right)} \quad\text{for every}\quad 0 < \tau < 1/2 \quad\text{and}\quad 0 < \tau \leq \beta < 1. 
$$
Let us compare these bounds with the explicit formulae given by Pollicot-Weiss in \cite{pollicot.weiss.1994}:
\ 
\\
\\
\begin{itemize}
 \item $\tau = \beta$: conformal case,
 $$
 \dim_B(\Lambda) = \dH(\Lambda) = \dfrac{\log{2}}{\log(1/\beta)};
 $$
 \item $0 < \tau < \beta < 1/2$: 
 $$
 \dim_B(\Lambda) = \dH(\Lambda) = \dfrac{\log{2}}{\log(1/\beta)};
 $$
 \item $\beta = 2^{-\frac{1}{n}}, \ n \in \naturales$:
 $$
  \dim_B(\Lambda) = \dH(\Lambda) = \dfrac{\log(2\beta/\tau)}{\log(1/\tau)};
 $$
 \item there exists a $0 < \gamma < 1$ such that for almost every $\gamma < \beta < 1$:
 $$
  \dim_B(\Lambda) = \dH(\Lambda) = \dfrac{\log(2\beta/\tau)}{\log(1/\tau)};
 $$
 \item $\beta > 1/2$:
 $$
  \dim_B(\Lambda) = \dfrac{\log(2\beta/\tau)}{\log(1/\tau)} \quad \beta > 1/2.
 $$
\end{itemize}
Moreover, they conjecture that, for almost every $\tau$ and $\beta$ with $0 < \tau \leq \beta < 1$ and $\tau \leq 1/2$ it holds that $\dim_B(\Lambda) = \dH(\Lambda)$. Our bounds fits very well with these formulas, since:
$$
L = \dfrac{\log(2)}{\log\left(\frac{1}{\tau}\right)} \leq \dfrac{\log(2)}{\log\left(\frac{1}{\beta}\right)} \leq \dfrac{\log(2\beta/\tau)}{\log(1/\tau)} = U
$$
\textit{for every} $0 < \tau \leq \beta < 1$ and $\tau \leq 1/2$ so, for all the cases covered by the explicit formulas, it holds that
$$
L \leq \dH(\Lambda) \leq \underline{\dim}_B(\Lambda) \leq \overline{\dim}_B(\Lambda) \leq U
$$

\subsection{Bedford-McMullen carpets:} let  $\R = [0, 1] \times [0, 1]$ be the unit square, $1 < n < m$ and postive integers and define
$$
g_{ij}(x,y) =  
\begin{bmatrix}
 m^{-1} & 0 \\
 0      & n^{-1}
\end{bmatrix}
\begin{bmatrix}
x \\
y
\end{bmatrix}
+ 
\begin{bmatrix}
i/n \\
j/m
\end{bmatrix}
\quad\text{such that}\quad \R_{ij} = g_{ij}(\R), \ (i,j) \in S.
$$
$\mathcal{G} = \{g_{ij}: (i,j) \in S\}$ is a contractive non conformal IFS. The template of the construction os a collection $\{\R_{ij}\}$ of rectangles of sides $1/n \times 1/m$ labeled by some subset $S \subset \{1, \cdots n\} \times \{1, \cdots , m\}$. Figure \ref{figure(2)} is an example of the template of Bedford-McMullen carpet. There exist exact formulas for the box-counting dimension and Hausdorff dimension of Bedford-McMullen carpets, namely:
$$
\dim_B(\Lambda) = \dfrac{\log n_1}{\log n} + \dfrac{\log\left(\frac{1}{n_1}\sum_{i=1}^nN_i \right)}{\log m} \quad\text{and}\quad \dH(\Lambda) = \dfrac{\log\left(\sum_{i=1}^nN_i^{\log n/\log m}\right)}{\log n},
$$
where $N_i$ is the number of rectangles $R_{ij}$ chosen in the column $i$ of the template and $n_1$ is the number of columns with at least one rectangle. See \cite{falconer.survey.2013}. 

\begin{figure}
\centering
\includegraphics[width=0.7\linewidth]{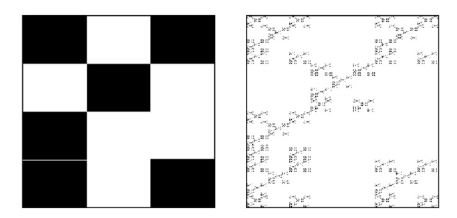}
\caption{}\label{figure(2)}
\end{figure}

It might happen however that, depending on the configuration of the template, this geometric construction might violate our conditions of regularity. For exmaple, the  template may have rectangles with overlapping borders or intersecting the border the unit square, so violating the border condition. Also, the  limit set might be volume reducible. This is the case, if all the rectangles projects onto the same interval on the $x$-axis since then the limit set might be a Cantor set contained in a vertical line. Moreover, the limit may converge to a Weierstrass-like graph, which is a connected set. We will call these cases \textit{degenerate}. However, we can perturb the initial template in order to get a regular construction, in the sense of our definitions. For this we choose a small $\epsilon > 0$ and introduce the rectangle
$$
\R^{\epsilon} = [\epsilon , 1-\epsilon] \times [\epsilon , 1-\epsilon]
$$
and define
$$
\R^{\epsilon}_{ij} = g_{ij}(\R^{\epsilon}), \ (i,j) \in S
$$
The IFS $\{g_{ij}\}$ acting upon the initial rectangle $\R^{\epsilon}$ give us a template satisfying the open and border condition. Also we can perturb the location of each rectangle given by moving it slightly horizontally or vertically in order to get a volume irreduble set. This \textit{perturbed regular geometric construction} $\mathcal{G}^{\epsilon} = \{g^{\epsilon}_{ij}\}$ acting upon $\R^{\epsilon}$ converges to a Cantor set $\Lambda_{\epsilon}$ which satisfies our hypothesis and is very near of the limit set $\Lambda$. As our bounds are robust this suggest that we can obtain, as we shall see below,  meaningful upper and lower bounds of the Hausdorff and box-counting dimension \textit{for every Bedford-McMullen carpet}, even if they are degenerate in the sense of definition \ref{definition.regular.constructions}. Let us calculate the upper and lower bounds given by Theorem \ref{main.thm.1}. To begin with,
$$
\|Dg\| = 1/n \quad\text{and}\quad \|\underline{Jg}\| = 1/nm,
$$
so giving
$$
\epsilon(\delta) = (1-\delta)\left(\dfrac{\|\underline{Jg}\|}{\quad\|Dg\|^{d - 1 + \delta}}\right)^{\frac{1}{1-\delta}} = (1-\delta)\left(\dfrac{n^{\delta}}{m}\right)^{1/1-\delta}.
$$
Thus, for $\delta = 0$,
$$
\epsilon = \dfrac{1}{m},
$$
as it is natural, since every rectangle of sides $1/n$ and $1/m$ can be decomposed into $m/n$ squares of side $1/n$ so, $\epsilon = 1/n$ is the characteristic scale of the system. Moreover,
$$
Df = 
\begin{bmatrix}
 n & 0 \\
 0 & m
\end{bmatrix}
$$
so
$$
\|Df\| = m, \quad Jf = nm \quad\text{and}\quad \lambda_0 = \lambda_1 = \log(nm).
$$
Our calculation for $\alpha$, the solution to the Bowen equation $P(f|\Lambda,-s\log Jf) = 0$ is
$$
\alpha = \dfrac{\log\left(\sum_{i=1}^nN_i\right)}{\log nm}.
$$
Thus,
$$
1 - \alpha = 1 - \dfrac{\log\left(\sum_{i=1}^nN_i\right)}{\log nm} = \dfrac{\log\left(\frac{nm}{\log\left(\sum_{i=1}^nN_i\right)}\right)}{\log nm}
$$
Therefore,
\begin{eqnarray*}
U & = & 2 + \dfrac{\lambda_0(1-\alpha)}{\log(\epsilon)} =  2 - \dfrac{\log\left(\dfrac{nm}{\log\left(\sum_{i=1}^nN_i\right)}\right)}{\log(m)}\\
& = & \dfrac{\log(m) - \log(n) + \log\left(\sum_{i=1}^nN_i\right)}{\log(m)}
=  \dfrac{\log\left(\frac{m}{n}\sum_{i=1}^nN_i\right)}{\log(m)}. 
\end{eqnarray*}
Now,
\begin{eqnarray*}
\dim_B(\Lambda) & = & \dfrac{\log n_1}{\log n} + \dfrac{\log\left(\frac{1}{n_1}\sum_{i=1}^nN_i \right)}{\log m}
 \leq  \dfrac{\log n_1}{\log n} - \dfrac{\log(n_1)}{\log(m)} + \dfrac{\log\left(\sum_{i=1}^nN_i\right)}{\log(m)}\\
& \leq & \dfrac{\log\left(\frac{m}{n}\sum_{i=1}^nN_i\right)}{\log(m)} = U
\end{eqnarray*}
for every $1 \leq n_1 \leq n$ since
$$
\dfrac{\log n_1}{\log n} - \dfrac{\log(n_1)}{\log(m)} \leq \dfrac{\log(m) - \log(n)}{\log(m)}, \quad\text{for every integer}\quad 1 \leq n_1 \leq n.
$$
To see this we can check that the function
$$
F(x) = \dfrac{\log(x)}{\log(n)} - \dfrac{\log(x)}{\log(m)}
$$
strictly increasing over the interval $[1, n]$. Indeed,
$$
F'(x) = \dfrac{1}{x}\left(\dfrac{1}{\log(n)} - \dfrac{1}{\log(m)} \right) > 0, \quad \forall \ x \in [1,n].
$$
Then, it attains its maximum over the interval $[1, n]$ at $x = n$:
$$
F(n) = 1 - \dfrac{\log(n)}{\log(m)} = \dfrac{\log(m) - \log(n)}{\log(m)} = \dfrac{\log(m/n)}{\log(m)}
$$
Then, $\dim_B(\Lambda) \leq U$ for every configuration with $n < m$. On the other hand, to calculate the lower bound we need to compute $\overline{\Sigma}$. As $\|Df\| = m$ and $\epsilon = 1/m$ we have
$$
\overline{\Sigma} = -\log((\|Df\|\epsilon)^{\overline{d_B}(\Lambda) + 2}\|Df\|^{-2}e^{\lambda_0}) = -\log(m^{-2}(nm)) = \log\left(\dfrac{m}{n}\right). 
$$
Therefore,
\begin{eqnarray*}
L & = & 2 + \dfrac{\overline{\Sigma} + \lambda_1(1-\alpha)}{\log(\epsilon)}  =  2 - \dfrac{\log\left(\dfrac{m}{n}\right) + \log\left(\frac{nm}{\log\left(\sum_{i=1}^nN_i\right)}\right)}{\log(m)} \\
& = & 2 - \dfrac{\log\left(\frac{m^2}{\log\left(\sum_{i=1}^nN_i\right)}\right)}{\log(m)}
 =  \dfrac{\log\left(\sum_{i=1}^nN_i\right)}{\log(m)}.
\end{eqnarray*}
We will show that, for every pair of integers $1 < n < m$,
$$
L = \dfrac{\log\left(\sum_{i=1}^nN_i\right)}{\log(m)} < \dfrac{\log\left(\sum_{i=1}^nN_i^{\log n/\log m}\right)}{\log n} = \dH(\Lambda)
$$ 
Let
$$
f(x) = x\log\left(\sum_{i=1}^nN_i^{1/x}\right).
$$
We will prove that $f(x)$ is strictly increasing for $x > 1$. We may suppose, without loss of generality, that $N_i > 1$ for every $i = 1, \cdots , n$. Then:
\begin{eqnarray*}
f'(x) & = & \log\left(\sum_{i=1}^nN_i^{1/x}\right) + x\dfrac{\sum_{i=1}^n(-1/x^2)\log(N_i)N_i^{1/x}}{\sum_{i=1}^nN_i^{1/x}}\\
& = & \log\left(\sum_{i=1}^nN_i^{1/x}\right) + \dfrac{\sum_{i=1}^n(-1/x)\log(N_i)N_i^{1/x}}{\sum_{i=1}^nN_i^{1/x}}\\
& = & \log\left(\sum_{i=1}^nN_i^{1/x}\right) + \dfrac{\sum_{i=1}^n\log(N_i^{-1/x})N_i^{1/x}}{\sum_{i=1}^nN_i^{1/x}}\\
& = & \dfrac{\sum_{i=1}^nN_i^{1/x}\log\left(\sum_{i=1}^nN_i^{1/x}\right) + \sum_{i=1}^n\log(N_i^{-1/x})N_i^{1/x}}{\sum_{i=1}^nN_i^{1/x}}\\
& = & \dfrac{\sum_{i=1}^nN_i^{1/x}\left(\log\left(\sum_{j=1}^nN_j^{1/x}\right) + \log\left(N_i^{-1/x}\right)\right)}{\sum_{i=1}^nN_i^{1/x}}\\
& = & \dfrac{\sum_{i=1}^nN_i^{1/x}\log\left(\frac{\sum_{j=1}^nN_j^{1/x}}{N_i^{1/x}}\right)}{\sum_{i=1}^nN_i^{1/x}} > 0,
\end{eqnarray*}
since,
$$
N_i^{1/x} > 0 \quad\text{and}\quad \dfrac{\sum_{j=1}^nN_j^{1/x}}{N_i^{1/x}} > 1, \quad i = 1, \cdots , n.
$$
In particular,
$$
\log\left(\sum_{i=1}^nN_i\right) < \dfrac{\log(m)}{\log(n)}\log\left(\sum_{i=1}^nN_i^{\log n/\log m}\right), \quad\text{for every}\quad 1 < n < m \in \mathbb{N}.
$$
This proves that $L \leq \dh(\Lambda)$. Hence,
$$
L = \dfrac{\log\left(\sum_{i=1}^nN_i\right)}{\log(m)} \leq \dH(\Lambda) \leq \dim_B(\Lambda)  \leq \dfrac{\log\left(\frac{m}{n}\sum_{i=1}^nN_i\right)}{\log(m)} = U 
$$
for every Bedford-McMullen carpet defined by a pair of integers $1 < n < m$ and any configuration of rectangles  
$\{\R_{ij}: (i,j) \in S\}$ with $S \subset \{1, \cdots n\} \times \{1, \cdots , m\}$. Let us compute these quantities for a particular example. See figure \ref{figure(2)}. In this case, $n = 3$ and $m = 4$ and the configuration of the template is described by the $0-1$ matrix, where $R_{ij} = 1$ iff $(i,j) \in S$:
$$
R = 
\begin{bmatrix}
1 & 0 & 1 \\
0 & 1 & 0 \\
1 & 0 & 0 \\
1 & 0 & 1
\end{bmatrix}
\quad\imply\quad 
N_1 = 3, \quad N_2 = 1 \quad N_3 = 2, \quad\text{and}\quad n_1 = 3.
$$
Thus, Bedford-McMullen formulas give $\dim_B(\Lambda) = 1.5000$ and $\dH(\Lambda) = 1.4866$, while our upper and lower bounds are
$$
U = 2 + \dfrac{\lambda_0(1-\alpha)}{\log\epsilon} = 1.5000 \quad\text{and}\quad 
L = 2 + \dfrac{\overline{\Sigma} + \lambda_1(1-\alpha)}{\log\epsilon} =  1.2925,
$$
with 
$$
\epsilon = 0.2500, \quad \alpha = 0.7211, \quad \overline\Sigma = 0.2877, \quad \lambda_0 = \lambda_1 = 2.4849.
$$
There are other examples as Gatzouras-Lalley, Bara\'nski and Feng-Wang carpets with explicit formulas which can be contrasted with the present results. See \cite{falconer.survey.2013}.

\section{Outline of the proof: methods}\label{section.methods}
 
We borrow the method from the work of Vanderlei Horita and Marcelo Viana \cite{vanderlei.viana}. There it is proved that certain nonuniformly expanding repeller $\Lambda$ has Hausdorff dimension strictly less that the dimension of ambient manifold using volumetric calculations. The starting step is the simple remark that the volume of a ball of radius $r$ is uniformly comparable with $r^d$. Consequently, $\sum_i\Vol(B_i)^\beta$ are uniformly comparable with $\sum_i\diam(B_i)^{d\beta}$ for covering by metric balls $B_i$. Then, we will concentrate into volume estimates of the coverings. Then a characteristic scale $0 < \epsilon < 1$  is chosen such that, for every atom $P \in \wp_n$ in the $n$-th stage of the construction, there exists a covering $\mathcal{B}_n(P)$ by $2\epsilon^n$-balls such that $\sum_{B \in \mathcal{B}_n(P)}\Vol(B)$ is bounded by $\Vol(P)$ up to a multiplicative constant only depending on the geometry of the ambient space. This is inequality (\ref{bound.1}). These coverings are a sort of decompositions of $P$ into pieces with uniformly bounded geometrical distortion, which allows to take care of the non conformality of the IFS.
 
Using (\ref{bound.1}) and thermodynamic formalism one prove that the sums $\sum_{B \in \mathcal{B}_n}\Vol(B)^\beta$ are bounded by a suitable geometric progression with ratio $\overline\theta(\beta) = e^{\lambda_0(\alpha-1)}\epsilon^{d(\beta-1)}$, up to a uniform multiplicative constant, where $\alpha$ is the solution to Bowen equation, $\epsilon$ is a our characteristic scale and $\lambda_0$ is given by the Jacobian of $f$as (\ref{defining.lambda0}). This our first fundamental inequality (\ref{first.fundamental.inequality}) (see section \ref{section.dB.estimations}). The ratio $\overline\theta(\beta)$ is a stricly decreasing function of $\beta$. Therefore, if $\overline{\beta}^*$ is the (unique) solution to the equation $\overline\theta(\beta) = 1$ then $\beta > \overline{\beta}^*$ implies that $\overline\theta(\beta) < 1$. Therefore, $\sum_{B \in \mathcal{B}_n}\Vol(B)^\beta \to 0^+$ exponentially, as $n \to +\infty$. This shows that $\overline{B}_{d\beta}(\Lambda) = 0$ for every $\beta > \overline{\beta}^*$ (see Appendix A for definitions). Thus, $\dH(\Lambda) \leq \overline{\dim_B}(\Lambda \leq \overline{\dim_B}(\Lambda) \leq d\overline{\beta}^*$. Solving for $d\overline{\beta}^*$ we get the upper bound
$$
U = d + \dfrac{\lambda_0(1-\alpha)}{\log\epsilon}.
$$

To bound from below the Hausdorff dimension it is a more delicate. We need to bound the corresponding $\beta$-sums $\sum_{B \in \mathcal{B}_0}\Vol(B)^\beta$ calculated over nonuniform $\zeta$-coverings by balls with variable radius $< \zeta$. We may suppose, without loss of generality, that $\mathcal{B}_0$ is a finite covering of $\Lambda$. The idea is to choose $n$ conveniently such that $2\epsilon^n$ is less than the Lebesgue number of $\mathcal{B}_0$ in such way that $\mathcal{B}_n$ is subordinated to the covering $\mathcal{B}_0$. For this is necessary to choose $n$ large enough, depending on the $\zeta$-covering $\mathcal{B}_0$ and we get that the sums 
$\sum_{B \in \mathcal{B}_0}\Vol(B)^\beta$ are bounded from below by $\sum_{B \in \mathcal{B}_n}\Vol(B)^\beta$, up to a multiplicative constant $C[\epsilon^{nd(1-\beta)}\zeta^{d(1-\beta)}]$, where $C > 0$ is a uniform constant depending only on the geometry of the ambient space. This is (\ref{second.fundamental.inequality}) proved in lemma \ref{lemma.second.fundamental.inequality}.

Now, $\mathcal{B}_n$ is too large a covering and we need a more efficient choice. We choose a suitable efficient covering $\mathcal{B}_n(\Lambda)$ of $\Lambda$ by $2\epsilon^n$-balls, starting with an efficient covering of $\Lambda \cap P$, for each piece $P \in \wp_n$ of the $n$-th stage of the construction $\Lambda_n$, using $\epsilon^n$ balls and then extending it to a covering by $2\epsilon^n$-balls of $P$, so forming a new covering $\mathcal{B}_n(\Lambda)$. Then we estimate $\rho(P)$ the fraction of $\Vol(P)$ which is covered by the $2\epsilon^n$-balls in  $\mathcal{B}_n(\Lambda \cap P)$, for each $P \in \wp_n$ and prove that the asymptotic rate of decay of $\rho(\wp_n(x))$ is a bounded measurable function $\Sigma(x)$, where $\wp_n(x)$ denotes the piece of $\wp_n$ containing $x$. This function is bounded from above (resp. from below) by suitable constant $\overline\Sigma$ (resp. $\underline\Sigma$). Thermodynamic formalism and volumetric methods allows to prove that, for every small $\gamma > 0$ and for sufficiently large $n$, depending on $\gamma$, the sums $\sum_{B \in \mathcal{B}_n(\Lambda)}\Vol(B)^\beta$ are bounded from below by $[e^{-\lambda_1(1-\alpha)}e^{-(\overline{\Sigma} + \gamma)}\epsilon^{d(\beta-1)}]^n$ up to a uniform constant. This is our third fundamental inequality (\ref{third.fundamental.inequality}) proved in lemma \ref{lemma.third.fundamental.inequality}.

Then, we use (\ref{second.fundamental.inequality}) and (\ref{third.fundamental.inequality}) to show that, for every \textit{sufficiently} small $\zeta$ and for every small $\gamma$ and every $\zeta$-covering ${\mathcal B}_0$ the sums $\sum_{B \in {\mathcal B}_0}\Vol(B)^{\beta} \geq C[e^{-n\lambda_1(1-\alpha)}e^{-(\overline{\Sigma} + \gamma)}\zeta^{d(\beta-1)}]$, for every sufficiently large $n$, depending of $\mathcal{B}_0$, where $C > 0$ is a uniform constant. This is our fourth fundamental inequality (\ref{fourth.fundamental.inequality}), proved in lemma \ref{lemma.fourth.fundamental.inequality}. 

If $\underline{\beta}^*$ is the solution to the equation $e^{-\lambda_1(1- \alpha)}e^{-(\overline{\Sigma} + \gamma)}\epsilon^{d(\beta -1)} = 1$ then, by a convenient choice of $n$, one can cancel the dependence on $\zeta$, the covering $\mathcal{B}_0$ and $n$ itself in our fourth fundamental inequality so getting a \textit{uniform lower bound $C > 0$} for the sums $\sum_{B \in \mathcal{B}_0}\Vol(B)^{\underline{\beta}^*}$, for every sufficiently small $\zeta$ and for every $\zeta$-covering $\mathcal{B}_0$ by balls of diameter $< \zeta$. Using the bounded relation between the volume and the radius of a ball we conclude that $\mathcal{H}_{d\underline{\beta}^*,\zeta}(\Lambda) \geq C > 0$ for every sufficiently small $\zeta$, for a uniform constant $C$ not depending on $\zeta$. Therefore, $\mathcal{H}_{d\underline\beta^*}(\Lambda) \geq C > 0$. Thus, by standard arguments of Hausdorff measure, we conclude that $\mathcal{H}_{d\beta}(\Lambda) = +\infty$ for every $\beta < \underline{\beta}^*$ so proving that $d\underline\beta^* \leq \dH(\Lambda)$. Now we solve for $d\underline{\beta}^*$ in the equation $e^{-\lambda_1(1- \alpha)}e^{-(\overline{\Sigma} + \gamma)}\epsilon^{d(\underline{\beta}^* -1)} = 1$ getting the lower bound
$$
L = d + \dfrac{\overline{\Sigma} + \gamma + \lambda_1(1-\alpha)}{\log\epsilon},
$$
which is valid for every small $\gamma > 0$. Passing to the limit $\gamma \to 0^+$ we get the lower bound in (\ref{main.inequality}). This completes the proof of Theorem \ref{main.thm.1}.

The plan of the exposition is as follows. In section \ref{section.dB.estimations} we prove the upper bound for the box-counting dimension introducing the characteristic scale $\epsilon > 0$ and the special coverings $\mathcal{B}_n$. In section \ref{section.dH.estimations} we introduce an efficient covering $\mathcal{B}_n(\Lambda)$ and prove the lower bound for the Hausdorff dimension.

\section{Proof of Theorem \ref{main.thm.1}: the upper bound for the box-counting dimension}\label{section.dB.estimations}

We first prove our upper bound. Let us recall the set up: let $g = \{g_i : i = 1, \cdots , s\}$ be finitely many contractions $g_i : N \to N$ from an embedded $d$-cube $N \subset M$ with positive volume into $N$. We suppose in addition that the boundary is piecewise smooth formed by finitely many codimension one embedded submanifolds. In particular, $\dim_B(\partial{N}) = d-1$. We also use the following volume estimates: there are two constants $0 < A_0 < A_1$, only depending on the Riemannian structure, such that
\begin{equation}\label{volume.diameter.constants.1}
A_0r^d \leq \Vol(B(x,r)) \leq A_1r^d \quad\text{for every}\quad x \in N \ r > 0
\end{equation}
Moreover, we get from (\ref{volume.diameter.constants.1}), that, for every $C > 0$,
\begin{equation}\label{volume.diameter.constants.2}
\dfrac{A_0C^d}{A_1}\Vol(B(x,r)) \leq \Vol(B(x,Cr)) \leq \dfrac{A_1C^d}{A_0}\Vol(B(x,r))
\end{equation}

The following lemma introduces the characteristic scale $0 < \epsilon < 1$. 

\begin{lemma}\label{lemma.1}
There exists $0 < \epsilon < 1$, $N_0 > 0$ a large positive number and a constant $C_0 > 0$, such that for every $n \geq N_0$  and every $n$-th generation atom $P \in \wp_n$, there exists a covering ${\mathcal B}_n(P)$ of $P$ by $2\epsilon^n$-balls such that,
\begin{equation}\label{bound.1}
\Vol(P) \leq \sum\limits_{B \in {\mathcal B}_n(P)}\Vol(B) \leq C_0\Vol(P) \quad\forall \ n \geq N_0.
\end{equation}
The constant $C_0$ only depends on $\dim_B(\partial{N})$, the box-counting dimension of the border $\partial{N}$, the volume $\Vol(N)$ and the geometry of the ambient manifold.
\end{lemma}

\begin{proof}
We will choose $0 < \epsilon < 1$ after establishing some conditions necessary to satisfy the required property in the statement of the lemma. We first recall that
$$
\dim_B(\partial{N}) = \lim_{\rho \to 0^+} \dfrac{\log\mathcal{N}(\partial{N},\rho)}{\log(1/\rho)} = d-1
$$
is the box-counting dimension of the border. Let any $0 < \delta < 1$, then
\begin{equation}\label{defining.d_0.delta}
\dim_B(\partial{N}) < d_0 := d - 1 + \delta < d
\end{equation} 
and there exists $\rho_0 > 0$ and $D_0 > 0$ such that:
$$
\forall \ 0 < \rho < \rho_0: \quad \partial{N} \quad\text{\textit{can be covered by at most}}\quad D_0\rho^{-d_0} \quad\text{\textit{balls of radius}}\quad \rho.
$$
Moreover, we choose $0 < \rho_0 < 1$ sufficiently small such that $\Lambda$ is contained in the complement of every covering of the border $\partial{N}$ by 
$\rho_0$-balls
$$
\Lambda \subset N - \bigcup_iB(z_i,\rho_0).  
$$
This is possible by the border condition on the construction. Our first condition to define $\epsilon$ is the following:
$$
\|Dg\|^{-n}2\epsilon^n < \rho_0 \quad\text{for sufficiently large $n$.}
$$
This can be proved whenever 
$$
0 < \|Dg\|^{-1}\epsilon < 1 \quad\text{and}\quad n \geq N_0(\epsilon) = \dfrac{\log(\rho_0/2)}{\log(\|Dg\|^{-1}\epsilon)}. 
$$
As we can seen the definition of $N_0 = N_0(\epsilon)$ depends on the choice of $\epsilon$ whenever $0 < \|Dg\|^{-1}\epsilon < 1$. We will fix $\epsilon > 0$ and $N_0$ at the end of the proof.

For every $n \geq N_0$, the border $\partial{N}$ can be covered by no more than $D_0(\|Dg\|^{-n}2\epsilon^n)^{-d_0}$ balls of radius $r = \|Dg\|^{-n}2\epsilon^n$. We will choose such a covering as follows: let $\{B(x_i,\|Dg\|^{-n}\epsilon^n)\}$ be maximal family of pairwise disjoint open $\|Dg\|^{-n}\epsilon^n$-balls covering the border. Then the open $\|Dg\|^{-n}2\epsilon^n$-balls covers $\partial{N}$ and we can therefore extract a minimal subcovering which will be denotated $\mathcal{B}_n(\partial{N}) = \{B(x_i,\|Dg\|^{-n}2\epsilon^n)\}$. For every $n > 0$, the open set 
$$
W_n(\partial{P}) = \bigcup_iB(g^n_P(x_i),2\epsilon^n) \quad\text{covers $\partial{P}$,}
$$
where $g^n_P$ is the branch of the IFS generating the $n$-atom $P \in \wp_n$. Moreover, for $n \geq N_0$ it has volume:
$$
\Vol(W_n(\partial{P})) \leq D_0A_1(\|Dg\|^{-n}2\epsilon^n)^{-d_0}(2\epsilon^n)^d
$$
due to (\ref{volume.diameter.constants.1}) and our bound on the cardinality of 
$\mathcal{B}_n(\partial{N})$. We can even suppose that,
$$
\Lambda \cap P \subset P- W_n(\partial{P}) \quad\text{for every}\quad n \geq N_0,
$$
since $\Lambda \subset N - \bigcup_iB(x_i,\rho_0)$ and $B_i \subset \bigcup_iB(x_i,\rho_0)$ for every $B_i \in \mathcal{B}_n(\partial{N})$ by our choice of $\epsilon$ and $N_0$.

Now we choose a maximal family of closed balls $\overline{B}(z_j,\epsilon^n)$ contained in $P-W_n(\partial{P})$ with pairwise disjoint interiors. Then, $\{B(z_j,2\epsilon^n)\}$ covers $P-W_n(\partial{P})$. 
Let 
$$
{\mathcal B}_n(P) = \{B(g^n_{P}(x_i),2\epsilon^n)\} \cup \{B(z_j,2\epsilon^n)\}.
$$
be the covering so constructed. 

By (\ref{volume.diameter.constants.2}),
$$
\sum_j\Vol(B(z_j,2\epsilon^n)) \leq \dfrac{A_12^d}{A_0}\sum_j\Vol(B(z_j,\epsilon^n)) \leq \dfrac{A_12^d}{A_0}\Vol(P), 
$$
since $\overline{B}(z_j,\epsilon^n) \subset P-W_n$ have pairwise disjoint interiors. Then,
\begin{eqnarray*}
\sum\limits_{B \in {\mathcal B}_n(P)}\Vol(B) & \leq & D_0A_1(\|Dg\|^{-n}2\epsilon^n)^{-d_0}(2\epsilon^n)^d + \dfrac{A_12^d}{A_0}\Vol(P)\\
                                        &  =   & D_0A_12^{d-d_0}(\|Dg\|^{d_0}\epsilon^{d-d_0})^n + \dfrac{A_12^d}{A_0}\Vol(P) \label{volume.total.1}.
\end{eqnarray*}
On the other hand,
$$
 \|\underline{Jg}\|^{n} \leq \dfrac{\Vol(P)}{\Vol(N)}.
$$
Indeed, $\Vol(P) \geq \inf_{x \in N}|Jg^n_{P}(x)|\Vol(N) \geq  \|\underline{Jg}\|^{n}\Vol(N)$. If 
$\epsilon > 0$ is chosen, such that:
\begin{equation}\label{defining.epsilon.1}
\|Dg\|^{d_0}\epsilon^{d-d_0} \leq \|\underline{Jg}\| 
\end{equation}
then, one have,
$$
\sum\limits_{B \in {\mathcal B}_n(P)}\Vol(B) \leq \left(\dfrac{A_12^{d-d_0}}{\Vol(N)} + \dfrac{A_12^d}{A_0}\right)\Vol(P),
$$
and therefore,
$$
C_0 = \dfrac{D_0A_12^{d-d_0}}{\Vol(N)} + \dfrac{A_12^d}{A_0}
$$
define a constant as required in (\ref{bound.1}). The solution to the inequality (\ref{defining.epsilon}) is:
$$
0 < \epsilon \leq \left(\dfrac{\|\underline{Jg}\|}{\ \ \|Dg\|^{d_0}}\right)^{\frac{1}{d-d_0}}.
$$
So, any $\epsilon > 0$ satisfying this inequality will work. Now, notice that $\|\underline{Jg}\| \leq \|Dg\|^d$. Therefore, using the definition of $d_0$ in (\ref{defining.d_0.delta})
$$
\left(\dfrac{\|\underline{Jg}\|}{\ \ \|Dg\|^{d_0}}\right)^{\frac{1}{d-d_0}} = \left(\dfrac{\|\underline{Jg}\|}{\quad\ \ \|Dg\|^{d - 1 + \delta}}\right)^{\frac{1}{1-\delta}} \leq (\|Dg\|^{1 - \delta})^{\frac{1}{1-\delta}} = \|Dg\|.
$$
Thus, given any $0 < \delta < 1$ we can define $\epsilon = \epsilon(\delta)$ as:
\begin{equation}\label{defining.epsilon(delta)}
\epsilon = (1-\delta)\left(\dfrac{\|\underline{Jg}\|}{\ \ \|Dg\|^{d - 1 + \delta}}\right)^{\frac{1}{1 - \delta}}
\end{equation}
where the factor $1-\delta$ its been chosen in order to adjust $\epsilon$ along the argument. With this choice of $\epsilon$ we have that $0 < \|Dg\|^{-1}\epsilon \leq (1-\delta) < 1$, for every $0 < \delta < 1$. Let $0 < \delta < 1$ arbitrary and $\epsilon = \epsilon(\delta)$ defined as (\ref{defining.epsilon(delta)}). For inequality (\ref{bound.1}) to hold one need,
 $$
  n \geq \dfrac{\log\rho_0}{\log(\|Dg\|^{-1}\epsilon)}.
  $$
As $\|Dg\|^{-1}\epsilon < 1-\delta$, one choose $n \geq N_0$, where $N_0 = N_0(\delta)$,
 \begin{equation}\label{defining.N_0}
 N_0 := \dfrac{\log\rho_0}{\log(1-\delta)} > \dfrac{\log\rho_0}{\log(\|Dg\|^{-1}\epsilon)}. 
 \end{equation}
\end{proof}

Let us collect some facts concerning the collection of coverings ${\mathcal B}_n(P)$ so constructed which will be used later:
\begin{enumerate}
 \item let ${\mathcal B}_n(\partial{P})$ be the collection of open $2\epsilon^n$-balls covering the border $\partial{P}$, then
 \begin{equation}\label{separation.border}
  \Lambda \subset P - W_n(\partial{P}), \quad\text{where}\quad W_n(\partial{P}) = \bigcup_{B \in {\mathcal B}_n(\partial{P})}B;
 \end{equation}
 \item the members of the \textit{reduced family}
 \begin{equation}\label{reduced.family}
  {\mathcal B}^*_n(P) = \{B(g^n_{P}(x_i),\epsilon^n/2)\} \cup \{B(z_j,\epsilon^n/2)\}
 \end{equation}
are pairwise disjoint:
\begin{itemize}
 \item $B(z_j,\epsilon^n) \cap B(z_{j'},\epsilon^n) = \emptyset$ for $j \neq j'$, by construction;
 \item $B(g^n_{P}(x_i),\epsilon^n/2) \cap B(z_j,\epsilon^n) = \emptyset$ since $B(g^n_{P}(x_i),\epsilon^n/2) \subset W_n$ and $B(z_j,\epsilon^n) \subset P - W_n$ and 
 \item suppose, by contradiction, that $x \in B(g^n_{P}(x_i),\epsilon^n/2) \cap B(g^n_{P}(x_{i'}),\epsilon^n/2) \neq \emptyset$. Then, $d(g^n_{P}(x_i), g^n_{P}(x_{i'})) < \epsilon^n$ and then $d(x_i,x_{i'}) < \|Dg\|^{-n}\epsilon_n$. But $B(x_i,\|Dg\|^{-n}\epsilon_n) \cap B(x_{i'},\|Dg\|^{-n}\epsilon_n) = \emptyset$, by construction so $B(g^n_{P}(x_i),\epsilon^n/2) \cap B(g^n_{P}(x_{i'}),\epsilon^n/2) = \emptyset$. 
\end{itemize}
\end{enumerate}

Our next lemma uses the well known \textit{Bounded Distortion Property}:
\ 
\\
\\
\textit{There exists a constant $C > 0$ such that, for every $n > 0$ and for every $x,y \in P$ and every 
$P \in \wp_n$,
\begin{equation}\label{bounded.volume.distorsion}
 \left|\dfrac{Jf^n(x)}{Jf^n(y)} - 1\right| \leq Cd(f^n(x),f^n(y))
\end{equation}
}
\ 
\\
\begin{lemma}\label{lemma.2}
There exist a constant $C_1 > 1$ such that, for every $x \in \Lambda$:
\begin{equation}\label{bounds.2}
C_1^{-1}e^{-\lambda_1n} \leq \Vol(\wp_n(x)) \leq C_1e^{-\lambda_0n} \quad\forall \ n > 0,
\end{equation} 
where, $0 < \lambda_0 < \lambda_1$
$$
\lambda_0 = \log\left(\inf_{x \in \Lambda}Jf(x)\right) \quad\text{and}\quad \lambda_1 = \log\left(\inf_{x \in \Lambda}Jf(x)\right)
$$
\end{lemma}

\begin{proof}
As $f^n(P) = N$, for every $P \in \wp_n$ then, by the bounded distortion property, there exists a constant $C_1 > 1$, only depending on the volume distortion of $f$ and the volume of $N$, such that
$$
C_1^{-1} \leq \Vol(\wp_n(x))Jf^n(x) \leq C_1,
$$
for every $x \in \Lambda$ and for every $n > 0$. The estimate (\ref{bounds.2}) follows inmediatly.
\end{proof}

\begin{lemma}\label{lemma.3}
Let $s = \alpha$ be the solution to the Bowen equation $P(f|\Lambda,-s\log(Jf)) = 0$. Then,
\begin{equation}\label{bounds.volume.partitions}
0 < \inf\limits_{n > 0}\sum\limits_{P \in \wp_n}\Vol(P)^{\alpha} \leq \sup\limits_{n > 0}\sum\limits_{P \in \wp_n}\Vol(P)^{\alpha} < +\infty.
\end{equation}
\end{lemma}
\begin{proof}
 Let $\mu_{\alpha}$ be the equilibrium state for the H\"older continuous potential $\phi = -\alpha\log(Jf)$. This is a Gibbs measure. Then by the bounded distortion property and Gibbs property, we can find a constant $C_2 > 1$ such that:
 $$
 C_2^{-1} \leq \dfrac{\mu_{\alpha}(\wp_n(x))}{\Vol(\wp_n(x))^{\alpha}} \leq C_2
 $$
 for every $x \in \Lambda$ and every $n > 0$. Then (\ref{bounds.volume.partitions}) follows immediately by an standard argument, since $\wp_n$ is a decomposition into disjoint pieces at the $n$-stage of the construction.
\end{proof}

\begin{lemma}\label{lemma.d_B.upper bound}
Let $0 < \alpha \leq 1$ be the solution to the Bowen equation, $0 < \delta < 1$ and $0 < \overline{\beta}^* < 1$ be the solution to the equation
\begin{equation}\label{defining.d_B.upper.beta}
e^{\lambda_0(1-\alpha)}\epsilon^{d(1-\overline{\beta}^*)} = 1,
\end{equation}
where $\epsilon = \epsilon(\delta)$ is defined in (\ref{defining.epsilon(delta)}). Then,
\begin{equation}\label{d_B.upper.bound}
\underline{\dim_B}(\Lambda) \leq \overline{\dim_B}(\Lambda) \leq d\overline{\beta}^*.
\end{equation}
\end{lemma}

\begin{proof} 

By lemma \ref{lemma.1} there exists a uniform constant $C_0 > 0$ and a positive $\epsilon > 0$ such that for every $n > 0$ and for every atom of generation $n$, $P \in \wp_n$, there exits a finite covering 
${\mathcal B}_n(P)$ by $2\epsilon^n$-balls $B_i = B(x_i,2\epsilon^n))$ such that 
$\sum_{B \in {\mathcal B}_n(P)}\Vol(B) \leq C_0\Vol(P)$. We define 
$$
{\mathcal B}_n = \{B: B \in {\mathcal B}_n(P), \ P \in \wp_n\}
$$ 
the open covering by these $2\epsilon^n$-balls of $\Lambda_n = \bigcup_{P \in \wp_n}P$, the $n$-th stage in the construction of the limit $\Lambda$.

Let $0 < \delta < 1$ arbitrary and $\epsilon = \epsilon(\delta)$, as defined in (\ref{defining.epsilon(delta)}). We will prove that, for every $\beta > 0$, there exists a constant $C_3(\beta) > 0$ such that
\begin{equation}\label{first.fundamental.inequality}
\sum\limits_{B \in {\mathcal B}_n}\Vol(B)^{\beta} \leq 
C_3(\beta)(e^{\lambda_0(\alpha-1)}\epsilon^{d(\beta-1)})^n, \quad\forall \ n \geq N_0(\delta),
\end{equation}
where $N_0(\delta)$ was defined in (\ref{defining.N_0}). To prove this let $\beta > 0$. Then:
$$
\sum\limits_{B \in {\mathcal B}_n(P)}\Vol(B)^{\beta} \leq \sum\limits_{B \in {\mathcal B}_n(P)}\Vol(B)\max_{B \in {\mathcal B}_n(P)}\Vol(B)^{\beta-1}.
$$
Hence,
$$
\sum\limits_{B \in {\mathcal B}_n(P)}\Vol(B)^{\beta} \leq C_0\Vol(P)[A_1(2\epsilon^n)^d]^{\beta-1}, 
$$
for every $n \geq N_0(\epsilon)$, by (\ref{bound.1}). Therefore,
$$
\sum\limits_{B \in {\mathcal B}_n}\Vol(B)^{\beta} \leq C_0[A_12^d]^{\beta-1}[\epsilon^{d(\beta-1)}]^n\sum\limits_{P \in \wp_n}\Vol(P).
$$
On the other hand, by (\ref{bound.1}) in lemma \ref{lemma.1},
$$
\begin{array}{ccc}
\sum\limits_{P \in \wp_n}\Vol(P) & \leq & \sum\limits_{P \in \wp_n}\Vol(P)^{\alpha}\max_{P \in \wp_n}\Vol(P)^{1-\alpha}\\
                                 & \leq & C_1^{1-\alpha}e^{-n\lambda_0(1-\alpha)}\sum\limits_{P \in \wp_n}\Vol(P)^{\alpha}.
\end{array}
$$
for every $n > 0$, by (\ref{bounds.2}). Thus
$$
\sum\limits_{B \in {\mathcal B}_n}\Vol(B)^{\beta} \leq C_3(\beta)(e^{\lambda_0(\alpha-1)}\epsilon^{d(\beta-1)})^n,
$$
for every $n \geq N_0(\delta)$, where 
$$
C_3(\beta) = C_0C_1^{1-\alpha}[A_12^d]^{\beta-1}\sup_{n > 0}\sum\limits_{P \in \wp_n}\Vol(P)^{\alpha},
$$
does not depend on $n > 0$. Now, ${\overline\theta}(\beta) = e^{\lambda_0(\alpha-1)}\epsilon^{d(\beta-1)}$ is strictly decreasing in $\beta$, therefore $e^{\lambda_0(\alpha-1)}\epsilon^{d(\beta-1)} < 1$ for $\beta > \overline{\beta}^*$, where $\overline{\beta}^*$ is the unique solution to the equation
$$
e^{\lambda_0(\alpha-1)}\epsilon^{d(\overline{\beta}^*-1)} = e^{\lambda_0(1-\alpha)}\epsilon^{d(1-\overline{\beta}^*)} = 1,
$$
In particular, $\sum_{B \in {\mathcal B}_n}\Vol(B)^{\beta} \to 0^+$ exponentially as $n \to +\infty$ for 
$\beta > \overline{\beta}^*$ as $e^{\lambda_0(\alpha-1)}\epsilon^{d(\beta-1)} < 1$. Let $\zeta > 0$ and
$$
B_{\beta,\zeta}(\Lambda) = \inf_{\mathcal{U}}\sum_{B_i \in \mathcal{B}(\zeta)}\diam(B_i)^{\beta},
$$
where infimum is taken over the whole family of at most countable coverings $\mathcal{B}(\zeta) = \{B_i(\zeta)\}$ of $\Lambda$ by $\zeta$-balls. Then, if $2\epsilon^n < \zeta$, 
\begin{eqnarray*}
B_{d\beta,\zeta}(\Lambda) & \leq & \sum_{B \in {\mathcal B}_n}\diam(B)^{d\beta}
 =  \sum_{B \in {\mathcal B}_n}(4\epsilon^n)^{d\beta}\\
& \leq & A_0^{-1}2^{d\beta}\sum_{B \in {\mathcal B}_n}\Vol(B)^\beta
 \leq  A_0^{-1}2^{d\beta}C_3(\beta)[e^{\lambda_0(\alpha-1)}\epsilon^{d(\beta-1)}]^n
\end{eqnarray*}
Then, if $\beta > \overline{\beta}^*$,
\begin{eqnarray*}
\overline{B}_{d\beta}(\Lambda) & = & \limsup_{\zeta \to 0^+}B_{d\beta,\zeta}(\Lambda)
 \leq  \limsup_{n \to +\infty}C_4(\beta)[e^{\lambda_0(\alpha-1)}\epsilon^{d(\beta-1)}]^n = 0,
\end{eqnarray*}
where, $C_4(\beta) = A_0^{-1}2^{d\beta}C_3(\beta)$. Thus,
$$
\overline{\dim}_B(\Lambda) = \inf\{a > 0: \overline{B}_{a}(\Lambda) = 0 \} \leq d\overline{\beta}^*.
$$
Therefore,
$$
\underline{\dim_B}(\Lambda) \leq \overline{\dim_B}(\Lambda) \leq d\overline{\beta}^*.
$$
so proving (\ref{d_B.upper.bound}).  
\end{proof}

\begin{proof}[Proof of the upper bound in (\ref{main.inequality}), Theorem \ref{main.thm.1}]
In lemma \ref{lemma.d_B.upper bound} we fix an arbitrary $0 < \delta < 1$ and get that 
$$
\dH(\Lambda) \leq \underline{\dim_B}(\Lambda) \leq \overline{\dim_B}(\Lambda) \leq d\overline{\beta}^*,
$$
where $\overline{\beta}^*$ is the solution to the equation $e^{\lambda_0(1-\alpha)}\epsilon^{d(1-\overline{\beta}^*)} = 1$, with $\epsilon = \epsilon(\delta)$. Solving for $d\overline{\beta}^*$ in this equation we get
$$
d\overline{\beta}^* = d + \dfrac{\lambda_0(1-\alpha)}{\log\epsilon(\delta)}.
$$
This inequality holds for every $0 < \delta < 1$. Then, passing to the limit as $\delta \to 0^+$ we get that
$$
\dH(\Lambda) \leq \underline{\dim}_B(\Lambda) \leq \overline{\dim}_B(\Lambda) \leq d + \dfrac{\lambda_0(1-\alpha)}{\log\epsilon} = U
$$
where $\epsilon = \epsilon(0)$ is defined in (\ref{defining.epsilon}). This concludes the proof of the upper bound in Theorem \ref{main.thm.1}. 
\end{proof}

\section{A lower bound for the Hausdorff dimension}\label{section.dH.estimations}
To bound the Hausdorff dimension from below we will bound sums $\sum_{B \in \mathcal{B}_0}\Vol(B)^{\beta}$, where ${\mathcal B}_0$ is an arbitrary non uniform covering of $\Lambda$ by open balls of radius $ < \zeta$. By non uniform we mean that the radius of these balls can be variables. \textit{We shall call this type of coverings a 
$\zeta$-covering}, for short. Furthermore, we can assume without loss of generality that ${\mathcal B}_0$ is finite.

By lemma \ref{lemma.1}, for every $P \in \wp_n$ there exists a finite open cover 
$$
{\mathcal B}_n(P) = \{B(x_i(P),2\epsilon^n), B(y_j(P),2\epsilon^n)\}
$$ 
such that $\sum_{B \in {\mathcal B}_n(P)}\Vol(B) \leq C_0\Vol(P)$ for a suitable constant $C_0 > 0$.

The following lemma says that, given given an arbitrary $\zeta$-covering ${\mathcal B}_0$ 
the sum $\sum_{B^0_k \in {\mathcal B}_0}\Vol(B^0_k)^{\beta}$ is bounded from below for sufficiently large $n$ by $\sum_{B \in {\mathcal B}_n}\Vol(B)^{\beta}$ up to a multiplicative constant depending on $n$, $\beta$, $\zeta$, $d$ and the characteristic scale $\epsilon > 0$. 

\begin{lemma}\label{lemma.second.fundamental.inequality}
Given an arbitrary $\zeta$-covering ${\mathcal B}_0$, there exists a large $N_1(\mathcal{B}_0) > 0$ and 
$C_5(\beta) > 0$ such that
\begin{equation}\label{second.fundamental.inequality}
\sum\limits_{B^0_k \in {\mathcal B}_0}\Vol(B^0_k)^{\beta} \geq C_5(\beta)[\epsilon^{nd(1-\beta)}\zeta^{d(\beta - 1)}]\sum\limits_{B \in {\mathcal B}_n}\Vol(B)^{\beta}
\end{equation}
for every $n \geq N_1(\mathcal{B}_0)$. The constant $C_5(\beta)$ depends on $\beta$ but not on $n$ neither the covering ${\mathcal B}_0$.
\end{lemma}

\begin{proof}[Proof of lemma \ref{lemma.second.fundamental.inequality}: first part]
Let $\lambda({\mathcal B}_0) < \zeta$ be the Lebesgue number of ${\mathcal B}_0$. We can choose 
$N_2$ large enough such that $2\epsilon^n < \lambda({\mathcal B}_0) < \zeta$ for every $n > N_2$. For this it is sufficient to define
\begin{equation}\label{defining.N_1}
N_1(\mathcal{B}_0) = \dfrac{\log\lambda({\mathcal B}_0)}{\log\epsilon}
\end{equation}

Recall that ${\mathcal B}_n = \{ B \in {\mathcal B}_n(P): P \in \wp_n\}$ is a $2\epsilon^n$-covering of $\Lambda_n$, the $n$-th stage of the geometric construction. The elements of the \textit{reduced family} 
$$
{\mathcal B}^*_n(P) = \{B(x_i(P),\epsilon^n/2), B(y_j(P),\epsilon^n/2)\}
$$
are pairwise disjoint, by construction. As $n \geq N_2$, we can use the Lebesgue number property to arrange the members $B \in {\mathcal B}_n$ as follows:
$$
\begin{array}{ccc}
{\mathcal B}_n(B^0_1) & := & \{ B \in {\mathcal B}_n : B \subset B^0_1\}\\
{\mathcal B}_n(B^0_2) & := & \{ B \in {\mathcal B}_n - {\mathcal B}_n(B^0_1): B \subset B^0_2\}\\
                      & \vdots & \\
{\mathcal B}_n(B^0_l) & := &   \{ B \in {\mathcal B}_n - \bigcup_{l' < l}{\mathcal B}_n(B^0_{l'}): B \subset B^0_l\}\\     
\end{array}
$$
For each $B^0_k \in {\mathcal B}_0$ the family ${\mathcal B}_n(B^0_k)$ is made of collections of open balls in ${\mathcal B}_n(P)$, where $P \cap B^0_k \neq \emptyset$. 

We would like to compare the volume covered by $B \in {\mathcal B}_n(B^0_k)$ and $\Vol(B^0_k)$. We're looking for an inequality of the type
$$
\sum_{B \in {\mathcal B}_n(B^0_k)}\Vol(B) \leq C\Vol(B^0_k),
$$
for a suitable uniform $C > 0$. One idea is to use the inequality (\ref{volume.diameter.constants.2}) relating 
$\Vol(B(x,r))$ with $\Vol(B(x,Cr))$ and the fact that the reduced family ${\mathcal B}^*_n(P)$ is disjoint for each $P \in \wp_n$. Then, 
$$
 \sum_{B \in {\mathcal B}_n(B^0_k)}\Vol(B) \leq \dfrac{A_14^d}{A_0}
 \sum_{B \in {\mathcal B}^*_n(B^0_k)}\Vol(B) \leq \dfrac{A_14^d}{A_0}\Vol(B^0_k)
$$
where
$$
{\mathcal B}^*_n(B^0_k) = \{B(1/4): B \in {\mathcal B}_n(B^0_k)\}
$$
and 
$$
B(C) = B(x,Cr), \quad\text{for}\quad B = B(x,r).
$$
with $\Vol(B(C)) \leq \frac{A_1C^d}{A_0}\Vol(B)$. However, even though the $\epsilon^n/2$-balls of the reduced covering ${\mathcal B}^*_n(P)$ are disjoint for every $P$, for different atoms $P \neq Q \subset \wp_n$, there might be overlapping balls $B \cap B' \neq \emptyset$ with $B \in {\mathcal B}^*_n(P)$ and $B' \in {\mathcal B}^*_n(Q)$. The following lemma take care of these overlappings to get the claimed inequality.
\end{proof}

\begin{lemma}\label{disjoining.border.overlaps}
There exists a constant $D > 1$, only depending on the topological dimension of $\partial{N}$ with the following property: for every $B^0_k \in {\mathcal B}_0$ and for every 
$n \geq N_1(\mathcal{B}_0)$
\begin{equation}\label{volume.5}
 \sum_{B \in {\mathcal B}_n(B^0_k)}\Vol(B) \leq D\Vol(B^0_k).
\end{equation}
\end{lemma} 

\begin{proof}
To prove (\ref{volume.5}) we will show that for every $B^0_k \in {\mathcal B}_0$ and for every 
$n \geq N_1(\mathcal{B}_0)$ there exists a collection of disjoint $\epsilon^n/2$-balls 
$$
\{B_i\} \subset {\mathcal B}^{*}_n(B^0_k)
$$
such that
\begin{equation}\label{volume.6}
 \sum_{B \in {\mathcal B}_n(B^0_k)}\Vol(B) \leq D\sum_i\Vol(B_i).
\end{equation}

For any given $P \in \wp_n$, the $\epsilon^n/2$-balls $B(y_j(P),\epsilon^n/2)$ contained in $P - W_n(\partial{P})$, are disjoint. Moreover, by the open condition, for every $Q \neq P$, the corresponding interior balls $B(y_j(P),\epsilon^n/2)$ and  $B(y_k(Q),\epsilon^n/2)$ are pairwise disjoint, so the unique overlappings between elements in ${\mathcal B}^*_n(P)$ for different atoms $P \in \wp_n$ can occur near the border. As $\mathcal{B}_n(\partial{P})$ is a minimal $2\epsilon^n$-covering of the border of $P \in \wp_n$ in the $n$-th stage in the construction and the open neighborhoods $W_n(\partial{P})$ are disjoint by the open condition then 
$$
\mathcal{B}_n(\partial\wp_n) = \bigcup_{P \in \wp_n}\mathcal{B}_n(\partial{P})
$$
is a minimal covering of the border of the $n$-th stage in the construction, $\partial\wp_n = \bigcup_{P \wp_n}\partial{P}$. Then, as $\partial{\wp_n}$ is a finite collection of pieces homeomorphic to $\partial{N}$, we can find a constant $\tau > 1$, only depending on the topological dimension of $\partial{N}$ such that there are at most $\tau$-overlappings in the open covering $\mathcal{B}_n(\partial\wp_n)$. In particular, for every open $\epsilon^n/2$-ball $B_j \in \mathcal{B}^*_n(B^0_k)$ in the reduced family, we have
\begin{equation}\label{multiplicity}
\#\{B_i \in \mathcal{B}^*_n(B^0_k): B_i \cap B_j \neq \emptyset\} \leq \tau.
\end{equation}

Now, fix $B^0_k \in \mathcal{B}_0$. To construct the collection $\{B_i\}$ (\ref{volume.6}) we proceed as follows. First, we separate the $\epsilon^n/2$-balls in $\mathcal{B}^*_n(B^0_k)$ into two classes: 
\begin{itemize}
 \item $\mathcal{J}_0$ is a maximal family of $\epsilon^n/2$-balls $B \in \mathcal{B}^*_n(B^0_k)$ which are pairwise disjoint;
 \item $\mathcal{J}_1$ contains all the $\epsilon^n/2$-balls in $\mathcal{B}^*_n(B^0_k)$ having overlappings.
\end{itemize}

Each  $B \in \mathcal{J}_1$  must belongs to $\mathcal{B}^*_n(\partial\wp_n)$, the reduced family of $\epsilon^n/2$-balls associated to the covering of the border $\partial\wp_n$, as we argued before. One can prove without difficulty that
$$
\bigcup\{B_i \in \mathcal{B}^*_n(\partial\wp_n)\ : \ B_i \cap B_j \neq \emptyset\} \subset B_j(4). 
$$
Then, by our argument with the topological dimension and using the inequalities (\ref{volume.diameter.constants.2}) we have that
$$
\sum\limits_{B_i \cap B_j \neq \emptyset}\Vol(B_i)  \leq  \dfrac{A_1\tau 4^d}{A_0}\Vol(B_j))
$$
using the estimate (\ref{multiplicity}) and
$$
B_i \subset B_j(4) \quad\Rightarrow\quad \Vol(B_i) \leq \dfrac{A_1\tau 4^d}{A_0}\Vol(B_j).
$$
Then, pick $B_1 \in \mathcal{J}_1$ and remove all its overlappings. Take $B_2$ outside the overlappings of $B_1$, that is,
$$
B_2 \in \mathcal{J}_1 - \{B_i \in \mathcal{J}_1 \ : \ B_i \cap B_1 \neq \emptyset\}
$$
and remove the overlappings associated to $B_2$ and so on. After finitely many steps we get a maximal collection $\mathcal{J}^0_1$ of disjoint open $\epsilon^n/2$ balls $B_i \in \mathcal{J}_1$, obtained by removing their overlappings. Complete this collection with $B \in \mathcal{J}_0$ to get a maximal disjoint collection $\{B_i\}$ of $\epsilon^n/2$-balls in ${\mathcal B}^*_n(B^0_k)$ such that
\begin{eqnarray*}
 \sum\limits_{B \in {\mathcal B}_n(B^0_k)}\Vol(B) & \leq & \dfrac{A_1\tau 4^d}{A_0}\sum\limits_{B \in {\mathcal B}^*_n(B^0_k)}\Vol(B)\\
 & = & \dfrac{A_1\tau 4^d}{A_0}\left(\sum\limits_{B \in \mathcal{J}_0}\Vol(B) + \sum\limits_{B \in \mathcal{J}_1}\Vol(B)\right)\\
 & \leq & \dfrac{A_1\tau 4^d}{A_0}\left(\sum\limits_{B_i \in \mathcal{J}_0}\Vol(B) + 
 \dfrac{A_1\tau 4^d}{A_0}\sum\limits_{B_i \in \mathcal{J}^0_1}\Vol(B_i)\right)\\
 & \leq & \dfrac{A_1\tau 4^d}{A_0}\sum\limits_{B_i \in \mathcal{J}_0}\Vol(B) + 
 \left(\dfrac{A_1\tau 4^d}{A_0}\right)^2\sum\limits_{B_i \in \mathcal{J}^0_1}\Vol(B_i)\\
 & \leq & \max\left\{\dfrac{A_1\tau 4^d}{A_0}, \left(\dfrac{A_1\tau 4^d}{A_0}\right)^2\right\}\sum_i\Vol(B_i)\\
 & \leq & D\Vol(B^0_k)
\end{eqnarray*}
where
$$
D = \max\left\{\dfrac{A_1\tau}{A_0}, \left(\dfrac{A_1\tau}{A_0}\right)^2\right\},
$$
which only depends on the topological dimension of $\partial{N}$ and the geometry of the ambient space. 

\end{proof}

Now, we are ready to conclude the proof of lemma \ref{lemma.second.fundamental.inequality}.

\begin{proof}[Proof of lemma \ref{lemma.second.fundamental.inequality}: conclusion]
By lemma \ref{disjoining.border.overlaps} 
$$
\sum\limits_{B \in {\mathcal B}_n(B^0_k)}\Vol(B) \leq  D\Vol(B^0_k),
$$
Hence,
\begin{eqnarray*}
\sum\limits_{B \in {\mathcal B}_n}\Vol(B)^{\beta} & \leq & \max_{B \in {\mathcal B}_n}\Vol(B)^{\beta-1}
\sum\limits_{B \in {\mathcal B}_n}\Vol(B)\\
& = & \max_{B \in {\mathcal B}_n}\Vol(B)^{\beta-1}\sum\limits_{B^0_k \in {\mathcal B}_0}\sum\limits_{B \in {\mathcal B}_n(B^0_k)}\Vol(B)\\
& \leq & [A_12^d\epsilon^{nd}]^{\beta-1}D\sum\limits_{B^0_k \in {\mathcal B}_0}\Vol(B^0_k)\\
& \leq & [A_12^d\epsilon^{nd}]^{\beta-1}D\max_{B^0_k \in {\mathcal B}_0}\Vol(B^0_k)^{1-\beta}\sum\limits_{B^0_k \in {\mathcal B}_0}\Vol(B^0_k)^{\beta}\\
& \leq & [A_12^d\epsilon^{nd}]^{\beta-1}D[A_1\zeta^d]^{1-\beta}\sum\limits_{B^0_k \in {\mathcal B}_0}\Vol(B^0_k)^{\beta}\\
& \leq & [2^{d(\beta-1)}D][\epsilon^{nd(\beta-1)}\zeta^{d(1-\beta)}]\sum\limits_{B^0_k \in {\mathcal B}_0}\Vol(B^0_k)^{\beta}.
\end{eqnarray*} 
Therefore, given a $\zeta$-covering ${\mathcal B}_0$ we have for every $\beta > 0$
$$
\sum\limits_{B^0_k \in {\mathcal B}_0}\Vol(B^0_k)^{\beta} \geq C_5(\beta)[\epsilon^{nd(1-\beta)}\zeta^{d(\beta-1)}]\sum\limits_{B \in {\mathcal B}_n}\Vol(B)^{\beta},
$$
for every $n \geq N_1(\mathcal{B}_0)$, where
$$
C_5(\beta) = (2^{d(\beta-1)}D)^{-1}.
$$
\end{proof}

The covering ${\mathcal B}_n$ by $2\epsilon^n$-balls is not efficient in that there lot of balls not intersecting $\Lambda$. Then, to construct an efficient covering of $\Lambda$ from ${\mathcal B}_n$ we proceed as follows. We first recall how $\mathcal{B}_n$ was constructed. For every $n \geq N_0$ we choose a $\|Dg\|^{-n}2\epsilon^n$-covering of the border $\mathcal{B}_n(\partial{N})$ and generate a $2\epsilon^n$-covering of $\partial{P}$, for every $P \in \wp_n$ defining an open neighborhood $W_n(\partial{P})$ of its border. Then we choose a $2\epsilon^n$-covering of the complement $P - W_n(P)$ of this neighborhood of the border: fix a maximal family of non-overlapping open $\epsilon^n$-balls $\{B(z_j,\epsilon^n)\}$ inside the closed set $P - W_n(P)$ so that $\{B(z_j,2\epsilon^n)\}$ is on open covering of $P - W_n(P)$. This choice was completely arbitrary whenever the non-overlapping condition is fullfiled. So let us choose now this collection as follows: let us start, for every $P \in \wp_n$, with a \textit{maximal collection of non-overlapping open $\epsilon^n$-balls $\{B(y_k,\epsilon^n)\}$ with $y_k \in \Lambda \cap P$ contained in $P - W_n(P)$}. This is possible since $\Lambda \cap P \subset P - W_n(P)$, by our border condition and the choice of $n \geq N_0$. Then, we complete with a maximal collection 
$\{B(z_j,\epsilon^n)\}$ of non-overlapping open $\epsilon^n$-balls in $P-W_n(P)$ intersected with the complement of $\bigcup_k(y_k,2\epsilon^n)$. We call
$$
\mathcal{B}_n(\Lambda \cap P) = \{B(y_k,2\epsilon^n)\} \quad\text{and}\quad \mathcal{B}_n(P) = \{B(y_k,2\epsilon^n)\} \cup \{B(z_j,2\epsilon^n)\} \cup \{B(x_i,2\epsilon^n)\},
$$
where $\{B(x_i,2\epsilon^n)\}$ is the covering of the border $\partial{P}$ giving rise to the open neighborhood 
$W_n(\partial{P})$. Then define 
$$
{\mathcal B}_n(\Lambda) = \bigcup_{P \in \wp_n}{\mathcal B}_n(\Lambda \cap P).
$$
This is a better covering than ${\mathcal B}_n$ since we are omitting balls $B(2\epsilon^n)$ of our initial covering $\mathcal{B}_n$ which do not intersect $\Lambda$. Let us introduce
\begin{equation}\label{fraction.efficient.covering}
\rho(P) = \dfrac{\sum_{B \in {\mathcal B}_n(\Lambda \cap P)}\Vol(B)}{\Vol(P)}, 
\end{equation}
the fraction of the volume of $P$ which is covered by the $2\epsilon^n$-balls in ${\mathcal B}_n(\Lambda \cap P)$.

\begin{proposition}\label{proposition.construction.Sigma}
There exists a semicontinuous function
\begin{equation}\label{definition.Sigma(x)}
\Sigma(x) = \limsup_{n \to +\infty}-\dfrac{\log(\rho(\wp_n(x)))}{n},
\end{equation}
and constants $-\infty < \underline{\Sigma} \leq \overline{\Sigma} < +\infty$, with $\overleftarrow{\Sigma} > 0$, such that,
\begin{equation}\label{bounding.Sigma(x)}
\underline{\Sigma} \leq \Sigma(x) \leq \overline{\Sigma}, \quad \forall \ x \in \Lambda.
\end{equation}
Here,
\begin{equation}\label{definition.lower.Sigma}
  \underline{\Sigma} = 
  \log\left((\|Df\|\epsilon)^{-\overline{d_B(\Lambda)}(\Lambda) + d}\text{min}(Df)^{-d}e^{\lambda_1}\right)^{-1}
 \end{equation}
 and
 \begin{equation}\label{definition.upper.Sigma}
  \overline{\Sigma} = \log\left((\|Df\|\epsilon)^{-\overline{d_B(\Lambda)} + d}\|Df\|^{-d}e^{\lambda_0}\right)^{-1}.
 \end{equation} 
 Here $\overline{d_B(\Lambda)}$ the upper box-counting dimension of $\Lambda$ and
 $$
 \text{min}(Df) = \inf_{x \in \bigcup_iN_i}\inf_{\|v\| \leq 1}\|Df(x)v\|, \quad v \in T_xM.
 $$
\end{proposition}

To prove proposition \ref{proposition.construction.Sigma} we need first a couple of lemmas.

\begin{lemma}\label{lemma.choosing.delta_0}
 Let $\|Df\| = \sup_{x \in \bigcup_iN_i}\|Df(x)\|$. There exists $\delta_0 > 0$ such that 
$\|Df\|\epsilon < 1$  for every $0 < \delta < \delta_0$. 
\end{lemma}
\begin{proof}
By definition,
$$
 \epsilon = (1-\delta)\left(\dfrac{\|\underline{Jg}\|}{\quad\quad\|Dg\|^{d - 1 + \delta}}\right)^{\frac{1}{1-\delta}},
 $$
 where
 $$
 \|\underline{Jg}\| = \min_i\inf_{x \in N}Jg_i(x) \quad\text{and}\quad \|Dg\| = \max_i\sup_{x \in N}Dg_i(x).
 $$
 We will use the singular values of $Dg_i(x)$, $0 < s^d_i(x) \leq \cdots \leq s^1_i(x) < 1$ to estimate the product $\|Df\|\epsilon$. Indeed, as
 $$
 Jg_i(x) = s^1_i(x) \cdots s^d_i(x) \quad\quad \|Dg_i(x)\| = s^1_i(x) \quad\quad Df(g_i(x)) = (s^d_i(x))^{-1}
 $$
 we are led to estimate:
 $$
 \Delta(x,\delta) = (1-\delta)\left(\dfrac{s^1_i(x) \cdots s^d_i(x)}{(s^1_i(x))^{d - 1 + \delta}}\right)^{\frac{1}{1-\delta}}(s^d_i(x))^{-1}.
 $$
 This is a continuous function of $x \in N$ and $\delta \geq 0$. Then, provided that at least one singular value $s^k_i(x) < s^1_i(x)$, which is certainly the case because we are dealing with a non conformal map, then:
$$
\Delta(x,0) = \dfrac{s^1_i(x) \cdots s^d_i(x)}{(s^1_i(x))^{d - 1}}(s^d_i(x))^{-1} = \dfrac{s^2(x)}{s^1(x)} \cdots \dfrac{s^{d-1}(x)}{s^1(x)} < 1.
$$
 Thus we can find, by continuity and a compacity argument, $\delta_0 > 0$, a small positive number, such that:
 $$
 \Delta(x,\delta) < 1 \quad\text{for every}\quad 0 < \delta < \delta_0, \quad \forall \ x \in \Lambda.
 $$
 Then, taking supremum and infimum properly, we conclude that $\|Df\|\epsilon < 1$, as we claimed.
\end{proof}
\ 
\\
\\
\textit{Choosing $\delta$}
\ 
\\
\\
\textit{From now on we will fix $\delta$
\begin{equation}\label{fixing.delta}
0 < \delta < \delta_0,
\end{equation}
where $\delta_0$ was chosen in lemma \ref{lemma.choosing.delta_0}, and the respective 
$\epsilon = \epsilon(\delta)$ without further mention}.
\ 
\\
\\
 \begin{definition}\label{definition.min(Df^n)}
 Let $f^n : f^{-n}(N) \to N$. We define,
 \begin{equation}
 \text{min}(Df^n) = \inf_{x \in f^{-n}(N)}\inf_{\|v\| \leq 1}\|Df^n(x)v\|. 
 \end{equation}
 \end{definition}
 
 In particular, $\|Df^n(x)v\| \geq \text{min}(Df^n)\|v\|$, for every $v \in T_xM$ and $x \in f^{-n}N$.
 
 \begin{lemma}\label{lemma.min(Df^n)}
 \begin{equation}\label{min(Df^n).supermultiplicative}
 \text{min}(Df^n) \geq (\text{min}(Df))^n, \quad \forall \ n \geq 1.
 \end{equation}
 Moreover, let $B(x,r) \subset f^{-n}(N)$. Then, 
 \begin{equation}\label{inclusion.iterate.ball}
 B(f^n(x),\text{min}(Df^n)r) \subset f^n(B(x,r)). 
 \end{equation}
 \end{lemma}
 \begin{proof}
 To prove (\ref{min(Df^n).supermultiplicative}) we argue by induction. For $n =1$ the inequality is trivial. Then let $n > 1$ and suppose that (\ref{min(Df^n).supermultiplicative}) holds for every $1 \leq k < n$. Then, let $x \in f^{-n}N$:
 \begin{eqnarray*}
  \inf_{\|v\| \leq 1}\|Df^n(x)v\|  & = & \inf_{\|v\| \leq 1}\|Df^{n-1}(f(x))Df(x)v\| \geq 
 \inf_{\|v\| \leq 1}\text{min}(Df^{n-1})\|Df(x)v\| \\
 & \geq  & \text{min}(Df^{n-1})\text{min}(Df) \geq (\text{min}(Df))^{n-1}\text{min}(Df) = (\text{min}(Df))^n,
 \end{eqnarray*}
 by the inductive hypothesis. Now, notice that $f^{-n}N = g^n(N) = \bigcup_{P \in \wp_n}P$. If $B(x,r) \subset f^{-n}(N)$ then there exists $P \in \wp_n$ such that $B(x,r) \subset P$ and then $f^k(B(x,r) \subset f^k(P) \subset N$ for every $k = 0, \cdots , n$. If $y \in B(x,r)$ let $\gamma = \gamma_{xy}(t)$ be a $C^1$ curve such that $\gamma(0) = x$, $ y = \gamma(1)$ and let 
 $\ell(\gamma)$ be the length of the curve $\gamma$. Then:
 $$
 \ell(f^n(\gamma)) = \int_0^1\|Df^n(\gamma(t))\gamma'(t)\|dt \geq \text{min}(Df^n)\int_0^1\|\gamma'(t)\|dt = \text{min}(Df^n)\ell(\gamma).
 $$
 Therefore,
 $$
 \ell(f^n(\gamma)) \geq \inf_{\gamma = \gamma_{xy}}\text{min}(Df^n)\ell(\gamma) = \text{min}(Df^n)d(x,y).
 $$
 As $f^n : f^{-n}N \to N$ is a diffeomorphism then every curve joining $f^n(x)$ to $f^n(y)$ is of the form $f^n(\gamma_{xy})$ for some $\gamma = \gamma_{xy}$. Thus,
 $$
 d(f^n(x),f^n(y)) \geq \text{min}(Df^n)d(x,y).
 $$
 Hence, $B(f^n(x),\text{min}(Df^n)r) \subset f^n(B(x,r))$, so proving \ref{inclusion.iterate.ball}. 
 \end{proof}
 
\begin{proof}[Proof of proposition \ref{proposition.construction.Sigma}]
 We are looking for upper and lower bounds for 
 $$
 \rho(\wp_n(x)) = \dfrac{\sum_{B \in {\mathcal B}_n(\Lambda \cap \wp_n(x))}\Vol(B)}{\Vol(\wp_n(x))}
 $$
 By lemma \ref{lemma.2}, there exist a constant $C_1 > 1$ such that, for every $x \in \Lambda$:
$$
C_1^{-1}e^{-n\lambda_1} \leq \Vol(\wp_n(x)) \leq C_1e^{-\lambda_0n} \quad\forall \ n > 0,
$$
where, $0 < \lambda_0 < \lambda_1$ are the numbers
$$
\lambda_0 = \log\left(\inf_{x \in \Lambda}Jf(x)\right) \quad\text{and}\quad \lambda_1 = \log\left(\inf_{x \in \Lambda}Jf(x)\right).
$$
 Then,
 \begin{eqnarray*}
 \rho(\wp_n(x)) & \geq & \dfrac{\#{\mathcal B}_n(\Lambda \cap \wp_n(x))A_0(2\epsilon^n)^d}{C_1e^{-\lambda_0n}} \\
 & = & \dfrac{A_0}{C_1}\#{\mathcal B}_n(\Lambda \cap \wp_n(x))(2\epsilon^n)^d
 e^{\lambda_0n},
 \end{eqnarray*}
 where $A_0$ comes from the bound $A_0r^d \leq \Vol(B(x,r)) \leq A_1r^d$. See (\ref{volume.diameter.constants.1}). Now,
 \begin{equation}\label{bound.cardinality.1}
  \#\mathcal{B}_n(\Lambda \cap \wp_n(x))) \geq \mathcal{N}(\Lambda,\|Df\|^n2\epsilon^n).
 \end{equation}
To see this we first notice that $f^n(\Lambda \cap P) = \Lambda$. On the other hand, for every $2\epsilon^n$-ball $B(y_k,2\epsilon^n)$ contained in the reduced covering $\mathcal{B}_n(\Lambda \cap \wp_n(x)))$ we have that $f^n(B(y_k,2\epsilon^n)) \subset B(f^n(y_k),\|Df\|^n2\epsilon^n)$. Thus, $\{B(f^n(y_k),\|Df\|^n2\epsilon^n)\}$ is a finite covering by open $\|Df\|^n2\epsilon^n$-balls with the same cardinality of $\mathcal{B}_n(\Lambda \cap \wp_n(x)))$. This shows (\ref{bound.cardinality.1}). Then,
\begin{eqnarray*}
\rho(\wp_n(x)) & \geq & \dfrac{A_0}{C_1}\#{\mathcal B}_n(\Lambda \cap \wp_n(x))(2\epsilon^n)^de^{\lambda_0n}
 \geq  \dfrac{A_0}{C_1}\mathcal{N}(\Lambda,\|Df\|^n2\epsilon^n)(2\epsilon^n)^de^{\lambda_0n}\\
& = & \dfrac{A_0}{C_1}\mathcal{N}(\Lambda,\|Df\|^n2\epsilon^n)(\|Df\|^n2\epsilon^n)^d\left(\dfrac{(2\epsilon^n)^d}{(\|Df\|^n2\epsilon^n)^d}\right)e^{\lambda_0n}\\
& = & \dfrac{A_0}{C_1}\mathcal{N}(\Lambda,\|Df\|^n2\epsilon^n)(\|Df\|^n2\epsilon^n)^d\left(\dfrac{e^{\lambda_0}}{\|Df\|^d}\right)^n
\end{eqnarray*}
 Thus, for every $n \geq 1$,
\begin{eqnarray*}
\log\rho(\wp_n(x)) & \geq & \log\left(\dfrac{A_0}{C_1}\right) + \log\mathcal{N}(\Lambda,\|Df\|^n2\epsilon^n) + d\log(\|Df\|^n2\epsilon^n) 
   + n\log\left(\dfrac{e^{\lambda_0}}{\|Df\|^d}\right)\\
& = & \log\left(\dfrac{A_0}{C_2}\right) + \log(\|Df\|^n2\epsilon^n)
\left(\dfrac{\log\mathcal{N}(\Lambda,\|Df\|^n2\epsilon^n)}{\log(\|Df\|^n2\epsilon^n)} + d\right) 
  + 
n\log\left(\dfrac{e^{\lambda_0}}{\|Df\|^d}\right)
\end{eqnarray*}
Hence, 
 \begin{eqnarray*}
   \liminf_{n \to +\infty}\dfrac{\log(\rho(\wp_n(x)))}{n} & \geq & 
   (-\overline{d_B(\Lambda)} + d)\log(\|Df\|\epsilon) + \log\left(\dfrac{e^{\lambda_0}}{\|Df\|^d}\right)\\
   & = & \log\left((\|Df\|\epsilon)^{-\overline{d_B(\Lambda)} + d}\left(\dfrac{e^{\lambda_0}}{\|Df\|^d}\right)\right)\\
   & = & \log\left((\|Df\|\epsilon)^{-\overline{d_B(\Lambda)} + d}\|Df\|^{-d}e^{\lambda_0}\right),
 \end{eqnarray*}
 where we use that
 \begin{eqnarray*}
 \liminf_{n \to +\infty}\dfrac{\log\mathcal{N}(\Lambda,\|Df\|^n2\epsilon^n)}{\log(\|Df\|^n2\epsilon^n)} = 
 -\limsup_{n \to +\infty}-\dfrac{\log\mathcal{N}(\Lambda,\|Df\|^n2\epsilon^n)}{\log(\|Df\|^n2\epsilon^n)}
  =  -\overline{d_B(\Lambda)}.
 \end{eqnarray*} 
 where $\overline{d_B(\Lambda)}$ is the upper box counting dimension of $\Lambda$, since $\|Df\|^n2\epsilon^n \to 0$ as $n \to +\infty$, as $\|Df\|\epsilon < 1$ by our choice of $0 < \delta < \delta_0$. Then,
\begin{eqnarray*}
\limsup_{n \to +\infty}-\dfrac{\log(\rho(\wp_n(x)))}{n} & = & -\liminf_{n \to +\infty}\dfrac{\log(\rho(\wp_n(x)))}{n}\\
& \leq & \log\left((\|Df\|\epsilon)^{-\overline{d_B(\Lambda)} + d}\|Df\|^{-d}e^{\lambda_0}\right)^{-1}
\end{eqnarray*}

On the other hand, as ${\mathcal B}_n(\Lambda \cap \wp_n(x)) =  \{B(y_k,2\epsilon^n)\}$, where $\{B(y_k,\epsilon^n)\}$ is a maximal family of non-overlapping open 
$\epsilon^n$-balls with $y_k \in \Lambda \cap \wp_n(x)$, then $\{B(y_k,\epsilon^n/2)\}$ are pairwise disjoint and does not cover $\Lambda \cap \wp_n(x)$. Now, 
$$
B(f^n(y_k), \text{min}(Df^n)(\epsilon^n/2)) \subset f^n(B(y_k,\epsilon^n/2))
$$
by (\ref{inclusion.iterate.ball}) in Lemma \ref{lemma.min(Df^n)} and then the set $\{B(f^n(y_k), \text{min}(Df^n)\epsilon^n/2)\}$ is a collection of disjoint balls centered in points of $\Lambda$ with cardinality $\#{\mathcal B}_n(\Lambda \cap \wp_n(x))$, which can be completed to a minimal covering of $\Lambda$ by $\text{min}(Df^n)(\epsilon^n/2)$-balls. Consequently,
 $$
 \#{\mathcal B}_n(\Lambda \cap \wp_n(x)) \leq \mathcal{N}(\Lambda, \text{min}(Df^n)(\epsilon^n/2)).
 $$
 Therefore, as $C_1^{-1}e^{-\lambda_1n} \leq \Vol(\wp_n(x))$
 \begin{eqnarray*}
 \rho(\wp_n(x))  & \leq & A_1C_1\#{\mathcal B}_n(\Lambda \cap \wp_n(x))(2\epsilon^n)^de^{n\lambda_1}\\
 & \leq & A_1C_1\mathcal{N}(\Lambda, \text{min}(Df^n)\epsilon^n/2)(2\epsilon^n)^de^{n\lambda_1}\\
 & = & A_1C_1\mathcal{N}(\Lambda, \text{min}(Df^n)\epsilon^n/2)(\text{min}(Df^n)\epsilon^n/2)^d
 \dfrac{(2\epsilon^n)^d}{(\text{min}(Df^n)\epsilon^n/2)^d}e^{n\lambda_1}\\
 & \leq & 4^dA_1C_1\mathcal{N}(\Lambda, \text{min}(Df^n)\epsilon^n/2)(\text{min}(Df^n)\epsilon^n/2)^d\left(\dfrac{e^{\lambda_1}}{\text{min}(Df)^d}\right)^n,
 \end{eqnarray*}
 using that $\text{min}(Df^n) \geq (\text{min}(Df))^n$. Hence, for every $n \geq 1$,
 \begin{eqnarray*}
\log(\rho(\wp_n(x)))  & \leq & \log(4^dA_1C_2) + \log(\mathcal{N}(\Lambda, \text{min}(Df^n)\epsilon^n/2)) + d\log(\text{min}(Df^n)\epsilon^n/2) \\
& & - dn\log(\text{min}(Df)) + n\lambda_1\\
& = & \log(2^dA_1C_2)  + \log(\text{min}(Df^n)\epsilon^n/2)
\left(\dfrac{\log(\mathcal{N}(\Lambda, \text{min}(Df^n)\epsilon^n/2))}{\log(\text{min}(Df^n)\epsilon^n/2)} + d\right) \\
& & - nd\log(\text{min}(Df)) +  n\lambda_1\\
& \leq & \log(2^dA_1C_2)  + \log(\|Df\|^n\epsilon^n/2)
\left(\dfrac{\log(\mathcal{N}(\Lambda, \text{min}(Df^n)\epsilon^n/2))}{\log(\text{min}(Df^n)\epsilon^n/2)} + d\right) \\ 
& & - nd\log(\text{min}(Df)) +  n\lambda_1,
 \end{eqnarray*}
 since $\text{min}(Df^n) \leq \|Df^n\| \leq \|Df\|^n$ and then 
 \begin{eqnarray*}
 \liminf_{n \to +\infty}\dfrac{\log(\rho(\wp_n(x)))}{n} & \leq & 
 (-\overline{d_B(\Lambda)} + d)\log(\|Df\|\epsilon) - d\log(\text{min}(Df)) + \lambda_1\\
 & = & \log\left((\|Df\|\epsilon)^{-\overline{d_B(\Lambda)} + d}\text{min}(Df)^{-d}e^{\lambda_1}\right),
 \end{eqnarray*}
 where we use again that
 $$
 \liminf_{n \to +\infty}\dfrac{\log(\mathcal{N}(\Lambda, \text{min}(Df^n)\epsilon^n/2))}{\log(\text{min}(Df^n)\epsilon^n/2)} = -\overline{d_B(\Lambda)},
 $$
 as we proved before, since $\|Df\|\epsilon < 1$, by our choice of $0 < \delta < \delta_0$ and then
 $$
 \text{min}(Df^n)\epsilon^n \leq \|Df^n\|\epsilon^n \leq (\|Df\|\epsilon)^n \to 0 \quad\text{as}\quad n \to 
 +\infty.
 $$
 Hence,
 \begin{eqnarray*}
\limsup_{n \to +\infty}-\dfrac{\log(\rho(\wp_n(x)))}{n} & = & -\liminf_{n \to +\infty}\dfrac{\log(\rho(\wp_n(x)))}{n}\\
& \geq & \log\left((\|Df\|\epsilon)^{-\overline{d_B(\Lambda)} + d}\text{min}(Df)^{-d}e^{\lambda_1}\right)^{-1}
\end{eqnarray*}  
This completes the proof of proposition \ref{proposition.construction.Sigma}.
\end{proof}

\begin{lemma}\label{lemma.third.fundamental.inequality}
Let $\alpha$ be the solution to the Bowen equation (\ref{bowen.equation}). There exist constants $C_6(\beta) > 0$ and $C_7(\beta) > 0$, independent of $n > 0$, such that, for every $\beta > 0$ ans for every small $\gamma > 0$, there exists a large $N_2 = N_2(\gamma) > 0$, such that
\begin{equation}\label{third.fundamental.inequality}
C_6(\beta)[e^{-\lambda_1(1-\alpha)}e^{-(\overline{\Sigma} + \gamma)}\epsilon^{d(\beta-1)}]^n \leq \sum\limits_{B \in {\mathcal B}_n(\Lambda)}\Vol(B)^{\beta} \quad\leq\quad C_7(\beta)[e^{-\lambda_0(1 - \alpha)}
e^{-\underline{\Sigma}}\epsilon^{d(\beta-1)}]^n,
\end{equation}
for every $n \geq N_2(\gamma)$, where  $\overline\Sigma$ and $\underline\Sigma$ are the constants found in proposition \ref{proposition.construction.Sigma}. 
\end{lemma}
\begin{proof}
 Let ${\mathcal B}_n(\Lambda \cap P)$ be the covering of $\Lambda \cap P$ defined in proposition \ref{proposition.construction.Sigma}. By definition,
$$
\sum\limits_{B \in {\mathcal B}_n(\Lambda \cap P)}\Vol(B) = \rho(P)\Vol(P).
$$
Then,
\begin{eqnarray*}
\rho(P)\Vol(P) & = & \sum\limits_{B \in {\mathcal B}_n(\Lambda \cap P)}\Vol(B) 
 =  \sum\limits_{B \in {\mathcal B}_n(\Lambda \cap P)}\Vol(B)^{\beta}\Vol(B)^{1-\beta}\\
& \leq & (A_1(2\epsilon^n)^d)^{1-\beta}\sum\limits_{B \in {\mathcal B}_n(\Lambda \cap P)}\Vol(B)^{\beta}\\
& =  & [A_12^d]^{(1-\beta)}\epsilon^{nd(1-\beta)}\sum\limits_{B \in {\mathcal B}_n(\Lambda \cap P)}\Vol(B)^{\beta},
\end{eqnarray*}
multiplying and dividing each term of the sum by $\Vol(B)^{\beta}$ and using (\ref{volume.diameter.constants.2}). Therefore,
$$
\sum_{P \in \wp_n}\rho(P)\Vol(P) \leq [A_12^d]^{(1-\beta)}\epsilon^{nd(1-\beta)}\sum\limits_{B \in {\mathcal B}_n(\Lambda)}\Vol(B)^{\beta}.
$$
By the definition of $\Sigma(x)$ (\ref{definition.Sigma(x)}) as an upper rate of decaying and that 
$\underline{\Sigma} \leq \Sigma(x) \leq \overline{\Sigma}$ for every $x \in \Lambda$, we have that for every small $\gamma > 0$ there exists $N_2 = N_2(\gamma) > 0$ such that,
\begin{equation}\label{defining.N_2(gamma)}
\underline{\Sigma} \leq -\dfrac{\log\rho(P)}{n} \leq \overline{\Sigma} + \gamma, \quad\forall \ n \geq N_2(\gamma), \ \forall \ P \in \wp_n,
\end{equation}
that is:
\begin{equation}\label{bounds.rho(P)}
e^{-(\overline{\Sigma} + \gamma)n} \leq \rho(P) \leq e^{-\underline{\Sigma}n}, \quad\forall \ n \geq N_2(\gamma), \ \forall \ P \in \wp_n.
\end{equation}
Then,
$$
e^{-(\overline{\Sigma} + \gamma)n}\sum_{P \in \wp_n}\Vol(P) \leq [A_12^d]^{(1-\beta)}\epsilon^{nd(1-\beta)}\sum\limits_{B \in {\mathcal B}_n(\Lambda)}\Vol(B)^{\beta}, \quad\forall \ n \geq N_2(\gamma).
$$
Now, multiplying and dividing each term of the sum $\sum_{P \in \wp_n}\Vol(P)$ by the factor 
$\Vol(P))^{\alpha}$ and using the volume bounds (\ref{bounds.2}) we get
\begin{eqnarray*}
\sum\limits_{P \in \wp_n}\Vol(P) & = & \sum\limits_{P \in \wp_n}\Vol(P)^{\alpha}\Vol(P)^{1-\alpha}\\
& \geq & [e^{-\lambda_1n(1-\alpha)}]\sum\limits_{P \in \wp_n}\Vol(P)^{\alpha}.
\end{eqnarray*}
Combining these inequalities, we get the lower bound in (\ref{third.fundamental.inequality}):
\begin{eqnarray*}
\sum\limits_{B \in {\mathcal B}_n(\Lambda)}\Vol(B)^{\beta} & \geq & [A_12^d]^{(\beta-1)}\epsilon^{nd(\beta-1)}\sum\limits_{P \in \wp_n}\rho(P)\Vol(P)\\
                                                           & \geq & 
[A_12^d]^{(\beta-1)}\epsilon^{nd(\beta-1)}e^{-\lambda_1n(1 - \alpha)}e^{-(\overline{\Sigma} + \gamma)n}\sum\limits_{P \in \wp_n}\Vol(P)^{\alpha}\\
                                                           & = & 
C_6(\beta)[e^{-\lambda_1(1-\alpha)}e^{-(\overline{\Sigma} + \gamma)}\epsilon^{d(\beta-1)}]^n,
\end{eqnarray*}
for every $n \geq N_2(\gamma)$, where
$$
C_6(\beta) = [A_12^d]^{\beta-1}\inf_{n > 0}\sum\limits_{P \in \wp_n}\Vol(P)^{\alpha}.
$$
Similarly,
$$
[A_02^d]^{(1-\beta)}\epsilon^{nd(1-\beta)}\sum\limits_{B \in {\mathcal B}_n(\Lambda \cap P)}\Vol(B)^{\beta} \quad\leq\quad  \rho(P)\Vol(P)
$$
and then
\begin{eqnarray*}
[A_02^d]^{(1-\beta)}\epsilon^{nd(1-\beta)}\sum\limits_{B \in {\mathcal B}_n(\Lambda)}\Vol(B)^{\beta} 
& \leq &  \sum_{P \in \wp_n}\rho(P)\Vol(P) \\
& \leq & [e^{-n\lambda_0(1-\alpha)}]e^{-n\underline{\Sigma}}\sum\limits_{P \in \wp_n}\Vol(P)^{\alpha},
\end{eqnarray*}
so that, for every $n \geq N_2(\gamma)$,
\begin{eqnarray*}
\sum\limits_{B \in {\mathcal B}_n(\Lambda)}\Vol(B)^{\beta} & \leq & [A_02^d]^{(\beta-1)}[e^{-\lambda_0(1-\alpha)}\epsilon^{d(\beta-1)}e^{-\underline{\Sigma}}]^n\sum\limits_{P \in \wp_n}\Vol(P))^{\alpha}\\
& \leq & C_7(\beta)[e^{-\lambda_0(1-\alpha)}\epsilon^{d(\beta-1)}e^{-\underline{\Sigma}}]^n,
\end{eqnarray*}
which is the upper bound in (\ref{third.fundamental.inequality}), being
$$
C_7(\beta) = [A_02^d]^{\beta-1}\sup_{n > 0}\sum\limits_{P \in \wp_n}\Vol(P)^{\alpha} > 0.
$$
\end{proof}

\begin{remark}\label{remark.upper.Hausdorff}
It can be proved that, if $s = \overline{\beta}^*$ is the unique solution to the equation
 \begin{equation}\label{Hausdorff.beta.upper.bound}
  e^{-\lambda_0(1- \alpha)}e^{-\underline{\Sigma}}\epsilon^{d({\overline{\beta}^*}^*-1)} = 1
 \end{equation}
 where $\alpha$ is the solution to the Bowen equation (\ref{bowen.equation}), then, 
 $$
 \dH(\Lambda) \leq d\overline{\beta}^*
 $$ 
For this we let $\zeta > 0$ be a small positive number. If $n > \max\{\log(\zeta/2)/\log\epsilon,N_2\}$ then 
$\mathcal{B}_n(\Lambda)$ will be a $\zeta$-covering and then, by (\ref{third.fundamental.inequality}),
\begin{eqnarray*}
 \mathcal{H}_{d\beta,\zeta}(\Lambda)  & \leq & \sum\limits_{B \in {\mathcal B}_n(\Lambda)}\diam(B)^{d\beta}\\
 & = & \sum\limits_{B \in {\mathcal B}_n(\Lambda)}(4\epsilon^n)^{d\beta}\\
 & \leq & A_0^{-\beta}2^{d\beta}\sum\limits_{B \in {\mathcal B}_n(\Lambda)}\Vol(B)^{\beta}\\
 & \leq & A_0^{-\beta}2^{d\beta}C_7(\beta)[e^{-\lambda_0(1-\alpha)}e^{-\underline{\Sigma}}\epsilon^{d(\beta-1)}]^n
\end{eqnarray*}
If $\beta = \overline{\beta}^*$ satisfies the equation (\ref{Hausdorff.beta.upper.bound}) then 
$$
\mathcal{H}_{d\overline{\beta}^*,\zeta}(\Lambda) \leq A_0^{-\beta}2^{d\beta}C_7(\beta) < +\infty
$$
As $\zeta > 0$ is arbitrary and the upper bound does not depends on $\zeta$ we conclude that $\mathcal{H}_{d\overline{\beta}^*}(\Lambda) < +\infty$. Then, by standard arguments of Hausdorff measure 
$\mathcal{H}_{d\beta}(\Lambda) = 0$ for every $\beta > \overline{\beta}^*$. Therefore, $\dH(\Lambda) \leq d\overline{\beta}^*$, as we claimed. Unfortunately, this is not that good an estimate since,
$$
d\overline{\beta}^* = d + \dfrac{\underline{\Sigma} + \lambda_0(1-\alpha)}{\log\epsilon} > d + \dfrac{\lambda_0(1-\alpha)}{\log\epsilon} = U,
$$
which is the upper box-counting dimension found in lemma \ref{lemma.d_B.upper bound} in section \ref{section.dB.estimations}. This occurs since, generally speaking, $\underline{\Sigma} < 0$. In our calculations with non conformal affine models we found that the upper bound $d\overline{\beta}^*$ is rather large, compared with the Hausdorff dimension.
\end{remark}

Now we use lemma \ref{lemma.second.fundamental.inequality} and our previous results to get a lower bound for the sums $\sum_{B \in {\mathcal B}_0}\Vol(B)^{\beta}$, for a $\zeta$-covering ${\mathcal B}_0$ with suffciently small $\zeta$. 

\begin{lemma}\label{lemma.fourth.fundamental.inequality}
 Let $\beta > 0$ be a positive number and $\gamma > 0$ small. Then, there exists a constant $C_8(\beta) > 0$ and 
 $\zeta_0  > 0$, a small positive number, such that for every $0 < \zeta < \zeta_0$, for every finite $\zeta$-covering  ${\mathcal B}_0$ of $\Lambda$ and for every $n > N_1(\mathcal{B}_0)$ it holds
 \begin{equation}\label{fourth.fundamental.inequality}
\sum\limits_{B \in {\mathcal B}_0}\Vol(B)^{\beta} \geq C_8(\beta)[e^{-n\lambda_1(1-\alpha)}e^{-n(\overline{\Sigma} + \gamma)}\zeta^{d(\beta-1)}].
\end{equation}
The constant $C_8(\beta)$ is uniform, only depends on the geometry of the ambient space.
\end{lemma}

\begin{proof}
To prove the inequality (\ref{fourth.fundamental.inequality}) we will combine the estimation in  (\ref{second.fundamental.inequality}) with (\ref{third.fundamental.inequality}). In (\ref{second.fundamental.inequality}) we bound the $\beta$-sums over an arbitrary $\zeta$-covering 
$\mathcal{B}_0$ by $\beta$-sums over $2\epsilon^n$-covering $\mathcal{B}_n$, up to a multiplicative function 
$[\epsilon^{nd(\beta-1)}\zeta^{d(1-\beta)}]$, with the condition that $n \geq N_1(\mathcal{B}_0)$. On the other hand, for inequalities in lemma \ref{lemma.third.fundamental.inequality} one fix a small $\gamma > 0$ and choose $N_2(\gamma)$ such that (\ref{third.fundamental.inequality}) holds true one for every $n > N_2(\gamma)$. According to the definition (\ref{defining.N_1}) of $N_1(\mathcal{B}_0)$ we can choose $\zeta$ sufficiently small such that
$$
n > N_1(\mathcal{B}_0) = \dfrac{\log\lambda({\mathcal B}_0)}{\log\epsilon} > N_2(\gamma),
$$
for every $\zeta$-covering ${\mathcal B}_0$ with $0 < \zeta < \zeta_0$. As $\lambda({\mathcal B}_0) < \zeta$ and
$$
\dfrac{\log\lambda({\mathcal B}_0)}{\log\epsilon} > \dfrac{\log\zeta}{\log\epsilon}
$$
it is suffcient to choose $0 < \zeta_0 < 1$ such that:
$$
\dfrac{\log\zeta}{\log\epsilon} > N_2(\gamma) \quad\text{for every}\quad 0 < \zeta < \zeta_0,
$$
that is,
\begin{equation}\label{defining.zeta_0}
\zeta_0 := \exp\left(\dfrac{N_2(\gamma)\log\epsilon}{2}\right).
\end{equation}
Then, if $0 < \zeta < \zeta_0$ and $n > N_1(\mathcal{B}_0)$ we get that
\begin{eqnarray*}
\sum\limits_{B^0_k \in {\mathcal B}_0}\Vol(B^0_k)^{\beta} & \geq & C_5(\beta)[\epsilon^{nd(1-\beta)}\zeta^{d(\beta - 1)}]\sum\limits_{B \in {\mathcal B}_n}\Vol(B)^{\beta}\\
& \geq & C_5(\beta)[\epsilon^{nd(1-\beta)}\zeta^{d(\beta - 1)}]\sum\limits_{B \in {\mathcal B}_n(\Lambda)}\Vol(B)^{\beta}
\end{eqnarray*}
for every $n \geq N_1({\mathcal B}_0)$, by (\ref{second.fundamental.inequality}). Also, by the lower bound in (\ref{third.fundamental.inequality})
$$
\sum\limits_{B \in {\mathcal B}_n(\Lambda)}\Vol(B)^{\beta} \geq C_6(\beta)[e^{-\lambda_1(1-\alpha)}e^{-(\overline{\Sigma} + \gamma)}\epsilon^{d(\beta-1)}]^n,
$$
for a suitable $C_6$ and $n \geq N_2(\gamma)$. Then,
$$
\begin{array}{ccccc}
\sum\limits_{B^0_k \in {\mathcal B}_0}\Vol(B^0_k)^{\beta} 
& \geq & 
C_5(\beta)C_6(\beta)[\epsilon^{nd(1-\beta)}\zeta^{d(\beta - 1)}][e^{-\lambda_1(1-\alpha)}e^{-(\overline{\Sigma} + \gamma)}\epsilon^{d(\beta-1)}]^n\\
& = & 
C_8(\beta)[e^{-n\lambda_1(1-\alpha)}e^{-n(\overline{\Sigma} + \gamma)}\zeta^{d(\beta-1)}], 
\end{array}
$$
for every $n \geq N_1(\mathcal{B}_0)$, where $C_8(\beta) = C_5C_6$.
\end{proof}
\ 
\\
\textbf{Proof of Theorem \ref{main.thm.1}: conclusion}
\ 
\\
\\
We choose $0 < \zeta < \zeta_0$ and a $\zeta$-covering $\mathcal{B}_0$ of $\Lambda$. We also fix a small $\gamma > 0$. Recall that we fixed a $0 < \delta < \delta_0$ in remark \ref{fixing.delta}. Now, let $\underline\beta^*$ the unique solution to the equation:
\begin{equation}\label{defining.beta^*}
e^{-\lambda_1(1- \alpha)}e^{-(\overline{\Sigma} + \gamma)}\epsilon^{d(\underline\beta^* -1)} = e^{\lambda_1(\alpha^*-1)}\epsilon^{d(\underline\beta^*-1)} = 1,
\end{equation}
where 
$$
\alpha^* = \alpha - \dfrac{\overline{\Sigma} + \gamma}{\lambda_1}. 
$$
Then,
$$
\sum\limits_{B \in {\mathcal B}_0}\Vol(B)^{\underline\beta^*} \geq C_8(\beta)[e^{n\lambda_1(\alpha^*-1)}\zeta^{d(\beta^*-1)}].
$$
As $\mathcal{H}_{d\underline\beta^*,\zeta}(\Lambda)$ is a non decreasing function of 
$\zeta$ as it goes to zero we may, without loss of generality, choose $0 < \zeta < \zeta_0$ such that $0 < \mathcal{H}_{d\underline\beta^*,\zeta}(\Lambda) < +\infty$ for, otherwise, $\mathcal{H}_{d\underline\beta^*}(\Lambda) = +\infty$ so proving that $d\underline\beta^* \leq \dH(\Lambda)$, which concludes the proof. Let us fix such $0 < \zeta < \zeta_0$ and $\mathcal{B}_0$ a $\zeta$-covering of $\Lambda$ and define
$$
n = \left[\dfrac{\log(\lambda(\mathcal{B}_0)/2)}{\log\epsilon}\right] + 1,
$$
where $[x]$ is the integer part of $x$ and $\lambda = \lambda(\mathcal{B}_0)$ is a Lebesgue number for 
$\mathcal{B}_0$. By our choice of $\zeta_0$ in (\ref{defining.zeta_0}) and the definition of $N_1$ (\ref{defining.N_1}), 
$$
\dfrac{\log(\lambda(\mathcal{B}_0)/2)}{\log\epsilon} > N_1(\mathcal{B}_0) > N_2(\gamma).
$$
In particular, $n > N_1(\mathcal{B}_0) > N_2(\gamma)$. Moreover, as $[x]$ is the integer floor of $x$ we have
$$
n \leq \dfrac{\log(\lambda(\mathcal{B}_0)/2)}{\log\epsilon} + 1 \quad\quad\text{and then}\quad\quad n\lambda_1(\alpha^* - 1) \geq \left(\dfrac{\log(\lambda(\mathcal{B}_0)/2)}{\log\epsilon} + 1\right)
\lambda_1(\alpha^*-1),
$$
since $\lambda_1(\alpha^*-1) < 0$, because 
$$
0 < \alpha \leq 1 \quad\text{and}\quad \dfrac{\overline{\Sigma} + \gamma}{\lambda_1} > 0.
$$
We will see that our choice of $n$ permits to eliminate the dependence on $n$, $\zeta$ and the covering $\mathcal{B}_0$ in (\ref{fourth.fundamental.inequality}). Indeed, 
\begin{eqnarray*}
\sum\limits_{B \in {\mathcal B}_0}\Vol(B)^{\underline\beta^*} & \geq & C_8(\underline\beta^*)\exp\left[n\lambda_1(\alpha^*-1) + d(\underline\beta^*-1)\log\zeta\right]\\
 & \geq & C_8(\underline\beta^*)\exp\left[\left(\dfrac{\log(\lambda/2)}{\log\epsilon} + 1\right)\lambda_1(\alpha^*-1) + d(\underline\beta^*-1)\log\zeta\right]
   \\
& = & C_8(\underline\beta^*)\exp\left[\lambda_1(\alpha^*-1) + \left(\dfrac{\log(\lambda/2)}{\log\epsilon}\right)\lambda_1(\alpha^*-1) + d(\underline\beta^*-1)\log\zeta]\right]\\
& = & C_8(\underline\beta^*)\exp\left[\lambda_1(\alpha^*-1) + \dfrac{\log\zeta}{\log\epsilon}\left(\dfrac{\log(\lambda/2)}{\log\zeta}\lambda_1(\alpha^*-1) + d(\underline\beta^*-1)\log\epsilon\right)\right].
\end{eqnarray*}
Now, notice that
\begin{eqnarray*}
\dfrac{\log(\lambda/2)}{\log\zeta}\lambda_1(\alpha^*-1) + d(\underline\beta^*-1)\log\epsilon & = &
\left(\dfrac{\log(\lambda/2)}{\log\zeta} -1\right)\lambda_1(\alpha^*-1) 
 +  [\lambda_1(\alpha^* - 1) + d(\underline\beta^*-1)\log\epsilon]\\                                                                                                & = &
\left(\dfrac{\log(\lambda/2)}{\log\zeta} -1\right)\lambda_1(\alpha^* - 1),
\end{eqnarray*}
adding and substracting terms and recalling that $\lambda_1(\underline\alpha-1) + d(\underline\beta^*-1\log\epsilon = 0$, by (\ref{defining.beta^*}). Then,
\begin{eqnarray*}
 \sum\limits_{B \in {\mathcal B}_0}\Vol(B)^{\underline\beta^*} & \geq & 
 C_8(\underline\beta^*)\exp\left[\lambda_1(\alpha^*-1) + \dfrac{\log\zeta}{\log\epsilon}\left(\dfrac{\log(\lambda/2)}{\log\zeta} -1\right)\lambda_1(\alpha^* - 1)\right]\\
& = & C_8(\underline\beta^*)\exp\left[\lambda_1(\alpha^* - 1) + \dfrac{\log(\lambda/2\zeta)}{\log\epsilon}\lambda_1(\alpha^* - 1)\right].
\end{eqnarray*}
Now, as $\log(\lambda/2\zeta) < \log(1/2)$ since $\lambda(\mathcal{B}_0) < \zeta$, we have
$$
\dfrac{\log(\lambda/2\zeta)}{\log\epsilon}\lambda_1(\alpha^* - 1) = -\dfrac{\log(\lambda/2\zeta)}{\log\epsilon}\lambda_1(1-\alpha^*) > \dfrac{\log{2}}{\log\epsilon}\lambda_1(1-\alpha^*).
$$
Therefore,
\begin{eqnarray*}
 \sum\limits_{B \in {\mathcal B}_0}\Vol(B)^{\underline\beta^*} 
& > & C_8(\underline\beta^*)\exp\left[\lambda_1(\alpha^* - 1) + \dfrac{\log{2}}{\log\epsilon}\lambda_1(1 - \alpha^*)\right]\\
& = & C_8(\underline\beta^*)\exp\left[-\lambda_1(1 - \alpha^*) + \dfrac{\log{2}}{\log\epsilon}\lambda_1(1-\alpha^*)\right]\\
& = & C_8(\underline\beta^*)\exp\left[\dfrac{\log(2/\epsilon)}{\log\epsilon}\lambda_1(1-\alpha^*)\right] = C_9(\beta^*),
\end{eqnarray*}
where $C_9(\beta^*)$ does not depends on the covering $\mathcal{B}_0$ neither $\zeta$. Thus,
$$
\sum\limits_{B \in {\mathcal B}_0}\Vol(B)^{\underline\beta^*} \geq C_9({\underline\beta}^*) > 0
\quad\text{and then}\quad
\sum\limits_{B \in {\mathcal B}_0}\diam(B)^{d\underline\beta^*} \geq C_{10}(\underline\beta^*) > 0,
$$
for every finite $\zeta$-covering ${\mathcal B}_0$ with $0 < \zeta < \zeta_0$, for a suitable constant,
$$
C_{10}(\underline\beta^*) = \dfrac{2^{d{\underline\beta^*}}C_9({\underline\beta}^*)}{A_1^{\underline\beta^*}} > 0,
\quad\text{using}\quad
\Vol(B^0_k)^{\underline\beta^*} \leq A_1^{\underline\beta^*}\left(\dfrac{\diam(B^0_k)}{2}\right)^{d\underline\beta^*}
\quad\text{by (\ref{volume.diameter.constants.1}).}
$$
The constant $C_{10}(\underline\beta^*)$ only depends on the geometry of the ambient manifold, the rates of volume expansion of $f$ and the rates of contraction of its inverse branches and several other fixed topological data of the construction. As $C_{10}(\underline\beta^*)$ is independent of the $\zeta$-covering $\mathcal{B}_0$ we conclude that
$$
0 < C_{10}(\underline\beta^*) \leq {\mathcal H}_{d\underline\beta^*,\zeta}(\Lambda) < +\infty,
$$
by our choice of $\zeta$. Therefore, 
$$
0 < C_{10}(\underline\beta^*) < {\mathcal H}_{d\underline\beta^*}(\Lambda) \leq +\infty
$$
since ${\mathcal H}_{d\underline\beta^*,\zeta}(\Lambda)$ is non-decreasing in $\zeta > 0$ and the lower bound $C_{10}$ does not depend on $\zeta$. Thus, by an standard argument for the Hausdorff measure, we conclude that
$$
{\mathcal H}_{d\beta}(\Lambda) =  +\infty \quad\text{for every}\quad \beta < \underline\beta^*.
$$
Therefore,  $d\underline\beta^* \leq \dH(\Lambda)$. Hence, \textit{for every small $\gamma > 0$} and $0 < \delta < \delta_0$:
$$
d\underline\beta^* = d + \dfrac{\lambda_1(1-\alpha^*)}{\log\epsilon(\delta)} = d + \dfrac{\overline{\Sigma} + \gamma + \lambda_1(1-\alpha)}{\log\epsilon(\delta)} \leq \dH(\Lambda).
$$
As these are continuous functions of $\gamma$ and $\delta$, passing to the limit as $\gamma$ and $\delta$ goes to zero, we get
$$
L = d + \dfrac{\overline{\Sigma} + \lambda_1(1-\alpha)}{\log\epsilon} \leq \dH(\Lambda),
$$
where $\epsilon = \epsilon(0)$ defined in the statement of Theorem \ref{main.thm.1}. This proves the lower bound in (\ref{main.inequality}), concluding the proof of Theorem \ref{main.thm.1}.

\section{Appendix A} 
\subsection{Box-counting and Hausdorff dimension}\label{subsection.dimension.theory}

Let $(X,d)$ a complete metric space and $Z \subset X$ a subset. We say that a countable collection of open sets $\mathcal{U} = \{U_i\}$ of $Z$ is a \textit{$\delta$-covering} if $Z \subset \bigcup_iU_i$ and $\diam(U_i) < \delta$ for every $i$. We define an outer measure:
$$
\mathcal{H}_{a,\delta}(Z) = \inf_{\mathcal{U}}\sum_{i=1}^{+\infty}\diam(U_i)^a
$$
where infimum is taken over all the $\delta$-coverings of $Z$. The outer measure $\mathcal{H}_{a,\delta}(Z)$ is a non decreasing function of $\delta$ so we define the \textit{$a$-Hausdorff measure of $Z$} as 
$$
\mathcal{H}_{a}(Z) = \lim_{\delta \to 0^+}\mathcal{H}_{a,\delta}(Z) = \sup_{\delta > 0}\mathcal{H}_{a,\delta}(Z) 
$$
since $\mathcal{H}_{a,\delta}(Z)$ is non decreasing in $\delta$. This is the well known Carath\'eodory's method. $\mathcal{H}_{a}$ is a Borel regular measure and can be equally defined using coverings by closed or convex sets. Moreover, it holds that, if $\mathcal{H}_{a}(Z) > 0$ the 
$\mathcal{H}_{s}(Z) = +\infty$ for every $s < a$.  Similarly so if $\mathcal{H}_{a}(Z) < +\infty$ then 
$\mathcal{H}_{s}(Z) = 0$ for every $s > a$. See \cite[Theorem 4.7, Chapter 4]{mattila}. This property permits to define the \textit{Hausdorff dimension of a subset $Z \subset X$} as:
$$
 \dH(Z) = \inf\{a > 0: \mathcal{H}_{a}(Z) = 0 \} = \sup\{a > 0 : \mathcal{H}_{a}(Z) = +\infty\}
$$
This is a \textit{dimensional characteristic} of the set. It satisfies the following well-known properties:
\begin{enumerate}
 \item $\dH(Z) \leq \dH(Y)$ if $Z \subseteq Y$;
 \item $\dH(\{point\}) = 0$ and 
 \item $\dH\left(\bigcup_{i=1}^{+\infty}Z_i\right) = \sup_{i > 0}\dH(Z_i)$.
\end{enumerate}
The Hausdorff dimension is invariant under Lipschitz continuous map. See for instance \cite{falconer} and \cite{mattila}. Another important dimensional characteristic of a set is the \textit{box countig dimension or limit capacity}: let $Z \subset X$ be a compact set and define:
$$
\mathcal{N}(X,\rho) = \inf\{\ \#\,\mathcal{U}: \ \mathcal{U} = \{B(x_i,\rho)\}, \ \text{finite covering of $Z$ by $\rho$-balls} \ B(x,\rho) \ \}
$$
Then,
$$
\overline{\dim_B}(X) = \limsup_{\rho \to 0^+} \dfrac{\log\,\mathcal{N}(X,\rho)}{\log(1/\rho)} \quad\text{and}\quad \underline{\dim_B}(X) = \liminf_{\rho \to 0^+} \dfrac{\log\,\mathcal{N}(X,\rho)}{\log(1/\rho)}
$$
are respectively the \textit{upper box-counting dimension} and \textit{lower box-counting dimension} 
of a set $Z$. The \textit{box-counting dimension}
$$
\dim_B(X) = \lim_{\rho \to 0^+} \dfrac{\log\,\mathcal{N}(X,\rho)}{\log(1/\rho)}
$$
is defined if the limit exists. Box-counting dimension can also be defined using the Carath\'eodory method: let $Z \subset X$ and $a > 0$ we define the measure
$$
B_{a,\rho}(Z) = \inf_{\mathcal{U}}\sum_{i}\diam(B(x_i,\rho))^a,
$$
where infimum is taken over the family of \textit{finite or countable} open coverings by $\rho$-balls $\mathcal{U} = \{B(x_i,\delta)\}$. We then define
$$
\overline{B}_a(Z) = \limsup_{\rho \to 0^+}B_{a,\rho}(Z)
$$
and
$$
\underline{B}_a(Z) = \limsup_{\rho \to 0^+}B_{a,\rho}(Z).
$$
It can be proved without difficulty, using that $B_{a,\rho}(Z) \asymp \mathcal{N}(Z,\rho)\rho^a$, up to a constant, that
$$
\overline{\dim_B}(Z) = \inf\{a > 0: \overline{B}_a(Z) = 0 \} =  \sup\{a > 0: \overline{B}_a(Z) = +\infty \}
$$
and
$$
\underline{\dim_B}(Z) = \inf\{a > 0: \underline{B}_a(Z) = 0 \} =  \sup\{a > 0: \underline{B}_a(Z) = +\infty \}
$$
See \cite[Chapter 2, Section 6, page 36]{pesin.dimension.1997}. If 
$\dim_B(\Lambda) = \underline{\dim_B}(Z) = \overline{\dim_B}(Z)$ then
$$
\dim_B(\Lambda) = \inf\{a > 0: B_a(Z) = 0 \} = \sup\{a > 0 : B_a(Z) = +\infty\},
$$
where $B_a(Z) = \lim_{\rho \to 0^+}B_{a,\rho}(Z)$. It follows from the definitions that,
$$
\dH(Z) \leq \underline{\dim_B}(Z) \leq \overline{\dim_B}(Z),
$$
since every covering by $\rho$-balls is a $\rho$-covering and therefore, $\mathcal{H}_{a,\rho}(Z) \leq \underline{B}_{a,\rho}(Z) \leq \overline{B}_{a,\rho}(Z)$ for every $a > 0$ and $\rho > 0$.

\subsection{Thermodynamic formalism and dynamical dimension}\label{subsection.thermodynamics.dimension}

Let $f : \bigcup_iU_i \to U$ be defined a piecewise $C^r$, $r > 1$ differentiable map defined as $f|U_i = g^{-1}_i$, where $U_i = g_i(U)$, $i = 1, \cdots , s$ and $N \subset U$ is the open neighborhood provided by the open condition. Then, 
$$
\Lambda = \bigcap_{n=0}^{+\infty}f^{-n}N
$$
is the locally maximal $f$-invariant subset of $f$ in $\bigcup_iN_i$. $\Lambda$ is the set of point which never escapes from the initial template $\{N_i\}$. This is the simplest model of a non linear hyperbolic repeller topologically conjugated to a full-shift on $s$-symbols: there are constants  $C > 0$ and $\lambda > 1$ such that:
$$
\|Df^n(x)v\| \geq C\lambda^n\|v\| \quad \forall v \in T_xM, \ x \in \Lambda
$$
and there exists a homeomorphism $h : B^+(s) \to \Lambda$ intertwining $f$ and the shift $\sigma$: $f \circ h = h \circ \sigma$, where $\sigma(\omega)(n) = \omega(n+1)$, $\omega\in B^+(s) = \{1, \cdots , s \}^{\enteros^+}$ with the product topology. In contrast with other approaches we do not use directly this symbolic dynamics. 

There exists a deep connection between the dimension theory of dynamical systems and thermodynamic formalism. Let us recall some crucial notions of the theory.

Let $(X,d)$ the a compact metric space and $f : X \to X$ a continuous self map. We say that a subset $E \subset X$ is $(\epsilon,n)$-separated if for every pair of points $x,y \in E$ there exists $0 \leq k < n$ such that $d(f^k(x),f^k(y)) > \epsilon$. The cardinal of a maximal $(\epsilon,n)$-separated subset is essentially the number of dynamically different orbits up to time $n$ with precision $\epsilon > 0$. The \textit{topological entropy}, a quantitative indicator of the complexity of the dynamics, is the rate of growing of dynamically different orbits up to finite arbitrarily small precision. The \textit{topological pressure} of a continuous potential $\phi$ is a generalization of the topological entropy and can be defined as a the rate of growing of \textit{weighted} dynamically different orbits up to finite time and small precision,
$$
P(\phi) = \lim_{\epsilon \to 0^+}\limsup_{n \to +\infty}\dfrac{1}{n}\log\left(\sup_E\sum_{x \in E}e^{S_n\phi(x)}\right),
$$
where supremum is taken over the family of $(\epsilon,n)$-separated subsets and $S_n\phi(x) = \sum_{k=0}^{n-1}\phi(f^k(x))$. If $\psi = 0$ we are left with the topological entropy definition $P(0) = h_{top}(f)$. See \cite{bowen.1975} and \cite{walters.1975}. This important quantity is invariant under topological conjugation and can be used as a dynamical indicator of the statistical properties of the system. The important \textit{variational principle por topological pressure} says
$$
P(\phi) = \sup_{\mu \in \mathcal{M}_f}\left\{h(\mu) + \int\phi{d\mu}\right\} 
$$
where $\mathcal{M}_f$ is the set of $f$-invariant Borel probabilites, $h(\mu)$ the \textit{Kolmogorov-Sinai} metric entropy and 
$$
P_{\mu}(\phi) = h(\mu) + \int\phi{d\mu}
$$
is the \textit{free energy}, also called the \textit{measure-theoretical pressure}. We say that $\mu \in \mathcal{M}_f$ is an \textit{equilibrium state} for the potential $\phi$ if it maximizes the free energy, that is:
$$
P(\phi) = h(\mu) + \int\phi{d\mu}. 
$$
An important question is to give sufficient conditions for a potential $\phi$ to have a unique or at most finitely many equilibrium states. These ideas came from the statistical physics on infinite one-dimensional gas whose phase state is modeled on the (bilateral) shift. We refer to \cite{bowen.1975} and \cite{keller} for a comprehensive introduction to the subject.

It is the case that the thermodynamic formalism of $C^r$ ($r > 1$) hyperbolic repellers as the maximal $f$-invariant subsets $\Lambda$ defined by a piecewise expanding $f: \bigcup_iN_i \to N$, which serve as models of the Cantor sets we are studying it is well understood. It can be proved that, for any continuous function $\phi : \Lambda \to \real$, the topological pressure is:
$$
 P(f|\Lambda, \phi) = \limsup_{n \to +\infty}\dfrac{1}{n}\log\left(\sum_{P \in \wp_n}\sup_{x \in P}\exp{S_n\phi(x)}\right) = \limsup_{n \to +\infty}\dfrac{1}{n}\log\left(\sum_{x \in Per_n(\Lambda)}\exp{S_n\phi(x)}\right),
 $$
where $\Lambda_n = \bigcup_{P \in \wp_n}P$ is the $n$-th stage of the construction and $Per_n(\Lambda)$ is the set of periodic points of prime period $n$. If $\phi$ is H\"older continuous, then there exists a unique equilibrium state $\mu_{\phi}$ for $\phi$. This equilibrium state is a \textit{Gibbs measure} in that there exists a constant $G > 1$ such that:
\begin{equation}\label{gibbs.property}
 G^{-1} \leq \dfrac{\mu_{\phi}(\wp_n(x))}{exp(-P(\phi) + S_n\phi(x)} \leq G
\end{equation}
for every $x \in \Lambda$ and for every $n > 0$. See \cite{bowen.1975} and \cite{keller}.

As pointed out in \cite{pesin.dimension.1997}, several pressure-like dynamical indicators constructed in several settings of the termodynamic formalism are dimensional indicators associated to a suitable Carath\'eodory's structure. This remark had originated a number of useful generalizations of the above thermodynamic formalism to include more general classes of potentials, as subadditive and almost subadditive sequences, giving rise to a rich theory for non conformal sets. See for instance \cite{barreira.gelfert.2011},  \cite{barreira.dimension.2011}, \cite{pesin.dimension.1997} and \cite{pesin.review.2010}. Let us show how thermodynamic formalism can be used to get dimensional indicators from the dynamics. Let
$$
\wp = \bigcup_{n \geq 1}\wp_n.
$$
be the generating net for $\Lambda$: for every $x \in \Lambda$ there exists a decreasing sequence $\wp_{n+1}(x) \subset \wp_n(x)$, $n \geq 1$, such that:
$$
\{x\} = \bigcap_{n=1}^{+\infty}\wp_n(x)
$$
where $\wp_n(x) \in \wp_n$ is the atom of generation $n$ containing $x$. We use the generating familiy $\wp$ and the volume as a set function to define a suitable \textit{Carath\'eodory's structure} with a dimensional indicator as follows. We say that a countable family of subsets $\mathcal{U} = \{P_i\}$, with $P_i \in \wp$, is a $(\wp,\delta)$-covering of $X$ if a) $X \subset \bigcup_iP_i$ and $\diam(P_i) < \delta$. Then we define an outer measure:
$$
 \mathcal{D}_{a,\delta}(X) = \inf_{\mathcal{U}}\sum_{i=1}^{+\infty}\Vol(P_i)^a
$$
where infimum is taken over the family of $(\wp,\delta)$-coverings. As usual, 
$\mathcal{D}_{a,\delta}(X)$ is non decreasing in $\delta$ so we define the following \textit{dynamical $a$-measure}:
$$
 \mathcal{D}_{a}(X) = \sup_{\delta > 0}\mathcal{D}_{a,\delta}(X).
$$
This is a Borel regular measure and its dimensional indicator is the \textit{dynamical dimension}
$$
 \dim_{\mathcal{D}}(X) = \inf\{a > 0:  \mathcal{D}_{a}(X) = 0 \} = \sup\{a > 0:  \mathcal{D}_{a}(X) = +\infty\}
$$
It holds out that $\alpha = \dim_{\mathcal{D}}(\Lambda)$ is precisely the solution to the Bowen equation
 \begin{equation}\label{bowen.equation}
  P(f|\Lambda, -\alpha\log(Jf)) = 0,
 \end{equation}
 where $Jf = |\det(Df)|$ is the Jacobian $f$ with respect to the Riemannian volume. Indeed, it is well known that, by the bounded distortion property of the volume under $f$ (see below), there exists a constant $A > 1$, uniform, such that
$$
A^{-1} \leq \Vol(P)Jf^n(x) \leq A \quad\forall \ x \in P, \ P \in \wp_n, \ n \geq 1.
$$
This is also known as the volume lemma. Therefore, $\Vol(P)^a \asymp \sup_{x \in P}\exp{S_n\phi(x)}$, meaning that the ratios of these quantities are uniformly bounded, independent of $P \in \wp_n$ and $n$, by suitable constants. Here $\phi(x) = -a\log{Jf}(x)$ is the \textit{geometric potential}. This is a H\"older continuous function, by our assumptions on the differentiability of the map. By the thermodynamic formalism there exists a unique ergodic Borel probability $\mu_{\alpha}$ supported on $\Lambda$ which is an equilibrium state for the geometric potential with parameter $\alpha$. Using the volume lemma and the Gibbs property we find a constant $B > 1$ such that
$$
B^{-1} \leq \dfrac{\mu(P)}{\Vol(P)^{\alpha}} \leq B \quad\forall \ P \in \wp.
$$
Then, we can prove that $0 < \mathcal{D}_{\alpha}(\Lambda) < +\infty$. In particular, 
$$
\dd(\Lambda) = \dfrac{h(\mu_{\alpha})}{\chi^+(\mu_{\alpha})},
$$
where $\chi^+(\mu) = \int\log{Jf}d\mu$ is the Lyapunov exponent of the Borel probability $\mu$. It holds that, if $f$ is conformal then $Df(x) = a(x)id_{T_xM}$ and $\alpha$ is a solution to the Bowen equation (\ref{bowen.equation}) then $\dH(\Lambda) = \dim_B(\Lambda) = d\alpha$ so, the Hausdorff and box-counting dimension are equal, up to a constant, to the dynamical dimension. The properties of the conformal sets stated in subsection \ref{subsection.conformal.sets} are all a consequence of this fact and the thermodynamics formalism.

\end{document}